\documentclass[12pt, reqno]{amsart}
\usepackage{amssymb,amsmath,amsopn,textcomp, thmtools}

\usepackage[backgroundcolor=white, bordercolor=blue, 
linecolor=blue]{todonotes}
\usepackage[normalem]{ulem}

\usepackage[a4paper,margin=1in,footskip=0.25in]{geometry}
\usepackage[ansinew]{inputenc}
\usepackage[english]{babel}
\usepackage{mathrsfs}
\usepackage{setspace}
\usepackage[titletoc]{appendix}
\usepackage[colorlinks=true, linkcolor=blue, citecolor=blue]{hyperref}

\usepackage{mathtools}
\mathtoolsset{showonlyrefs}

\usepackage{dsfont}
\usepackage{bbm}

\usepackage{graphicx}
\usepackage{color}
\usepackage{caption,rotating}

\newtheorem{theorem}{Theorem}[section]

\newtheorem{proposition}[theorem]{Proposition}
\newtheorem{lemma}[theorem]{Lemma}
\newtheorem*{theorem*}{Theorem}

\theoremstyle{definition}
\newtheorem{definition}[theorem]{Definition}

\newtheorem{remark}[theorem]{Remark}
\newtheorem*{remark*}{Remark}

\renewenvironment{proof}[1][Proof]{\noindent\textit{#1.} }{\hfill 
	\rule{0.5em}{0.5em}}

\addto\extrasenglish{}
\addto\extrasenglish{} 

\numberwithin{equation}{section}

\usepackage{color}

\newcommand{\cN}{\mathcal{N}}
\newcommand{\cA}{\mathcal{A}}
\newcommand{\cM}{\mathcal{M}}
\newcommand{\cR}{\mathcal{R}}
\newcommand{\cE}{\mathcal{E}}
\newcommand{\cF}{\mathcal{F}}
\newcommand{\cQ}{\mathcal{Q}}
\newcommand{\cD}{\mathcal{D}}

\newcommand{\cB}{\mathcal{B}}
\newcommand{\cZ}{\mathcal{Z}}
\newcommand{\cX}{\mathcal{X}}
\newcommand{\bP}{\mathbb{P}}
\newcommand{\bE}{\mathbb E}
\newcommand{\bR}{\mathbb R}
\newcommand{\bI}{\mathbb I}

\def\scD{{\mathscr D}}
\def\scE{{\mathscr E}}
\def\scF{{\mathscr F}}

\def\scJ{{\mathscr J}}
\def\scA{{\mathscr A}}

\def\wt{\widetilde}
\def\e{{\mathrm{e}}}
\def\rd{\mathrm{d}}

\newcommand{\N}{\mathbb{N}}
\newcommand{\R}{\mathbb{R}}

\newcommand{\Rd}{{\mathbb{R}^d}}

\newcommand{\Z}{\mathbb{Z}}
\newcommand{\E}{\mathbb{E}}
\newcommand{\1}{\mathbbm 1}

\newcommand{\eps}{\varepsilon}

\def\nn{\nonumber} 
\newcommand{\switch}{\epsilon}
\newcommand{\HHq}[2]{\big(H_{#1}^{#2}\big)}

\newcommand{\Tj}[2]{\sum_{j={#1}}^{#2} \lengthR_j}
\newcommand{\TT}[2]{\Theta({#1},{#2})}

\newcommand{\lengthR}{\theta}

\renewcommand{\d}{\mathrm{d}}

\author{Jaehoon Kang }
\email{jhnkang@snu.ac.kr}
\address{Department of Mathematical Sciences, Seoul National University, Building 27, 1 Gwanak-ro, Gwanak-gu, Seoul 08826, Republic of Korea}

\author{Moritz Kassmann}
\email{moritz.kassmann@uni-bielefeld.de}
\address{Fakult\"{a}t f\"{u}r Mathematik, Postfach 100131, D-33501 Bielefeld, Germany}

\title[Markov processes of direction-dependent type]
{Heat kernel estimates for Markov processes \\ of direction-dependent type}

\begin{document}
	
\thanks{
	Financial support of the German Science Foundation through the International Research Training Group Bielefeld-Seoul IRTG 2235 is gratefully acknowledged. JK was supported by BK21 SNU Mathematical Sciences Division. MK gratefully acknowledges support by the Institut de Math\'{e}matiques de Toulouse.}

\begin{abstract}
We prove sharp pointwise heat kernel estimates for symmetric Markov processes associated with symmetric Dirichlet forms that are local with respect to some coordinates and nonlocal with respect to the remaining coordinates. The main theorem is a robustness result like the famous estimate for the fundamental solution of second order differential operators, obtained by Donald G. Aronson. Analogous to his result, we show that the corresponding translation-invariant process and the one given by the general Dirichlet form share the same pointwise points. 
\end{abstract}
	
	\maketitle

\setcounter{tocdepth}{1}	
\tableofcontents

\section{Introduction}

This work is devoted to the study of heat kernel estimates for a certain class of Markov processes. In the language of Partial Differential Operators, the work is concerned with pointwise estimates for linear parabolic operators with bounded measurable coefficients. Let us recall the famous robustness result of Aronson \cite{Aro68}. The Gauss-Weierstrass kernel $g$ on $(0,\infty) \times \R^d$, defined by
\begin{align*}
g(t,x) =  \frac{t^{-d/2}}{(2\pi)^{d/2}} \,  e^{\frac{-|x|^2}{2t}} ,
\end{align*}
on the one hand, is the density function of the transition probability for the Brownian Motion. On the other hand, it is the fundamental solution of the heat equation in $\R^d$. Given a uniformly elliptic operator of the form $\mathcal{A} = \sum\limits_{i,j=1}^d  \frac{\partial}{\partial x_i} \big( a_{ij}(x) \frac{\partial}{\partial x_j} \big)$, a fundamental result of \cite{Aro68} says that the corresponding fundamental solution $p_{\mathcal{A}}$ satisfies the pointwise bounds
\begin{align}\label{eq:aronson-bounds}
C_2 t^{-d/2} e^{\frac{-c_2 |x-y|^2}{t}} \,\leq\, p_{\mathcal{A}}(t,x,y) \, \leq\, C_1 t^{-d/2} e^{\frac{-c_1 |x-y|^2}{t}} 
\end{align}
for some positive constants $C_i$, $c_i$. The estimate \eqref{eq:aronson-bounds} can be interpreted as a robustness result because it says that every non-degenerate elliptic partial differential operator shares the same pointwise upper and lower bounds. Results of this type have been studied intensively for many operators in various metric measure spaces.

\medskip

An important step has been made by establishing robustness results of the above kind for Markov jump processes in \cite{BaLe02} and \cite{ChKu03}. Let $p_J(t,x,y)$ denote the fundamental solution of the operator
\begin{align}\label{eq:gen-operator-nonlocal}
u \mapsto \partial_t u - 
\mbox{p.v.} 
\int_{\R^d} \big( u(y) - u(x) \big) J(x,y) \; \d y\,,
\end{align}
where $J(x,y)$ is symmetric and satisfies for some $\alpha \in (0,2)$ and $c_1, c_2 > 0$ the relation $c_1 |x-y|^{-d-\alpha} \leq J(x,y) \leq c_2 |x-y|^{-d-\alpha}$ for all $x \ne y$. Similar to \eqref{eq:aronson-bounds}, the authors show that there are two positive generic constants $c_1, c_2$ such that for all $t > 0$ and $x,y \in \R^d$ the two-sided estimate
\begin{align}\label{eq:aronson-nonlocal}
c_{1} t^{-d/\alpha} \big(1 \wedge 
\tfrac{t}{|x-y|^\alpha}\big)^\frac{d+\alpha}{\alpha} \leq p_J(t,x,y) \leq c_{2} t^{-d/\alpha} \big(1 \wedge 
\tfrac{t}{|x-y|^\alpha}\big)^\frac{d+\alpha}{\alpha}  
\end{align}
holds true. Thus, the fundamental solutions $p_J$ turns out to be comparable with the fundamental solution of the fractional heat operator $\partial_t + (-\Delta)^{\alpha/2} $. Or, in the language of Probability Theory, it is proved in \cite{ChKu03} that the transition density function $p_J$ of the process corresponding to $J$ is comparable to the  transition density function of the rotationally symmetric $\alpha$-stable process. 

\medskip

The aim of the present work is to establish a similar robustness result for integro-differential operators that might be of second order with respect to some coordinates and of any positive order between $0$ and $2$ with respect to the remaining coordinates. In this sense, the operators are mixed local-nonlocal operators. Note that these operators are of a much more complicated nature than one gets when considering a superposition of local and nonlocal operators as for $-\Delta + (-\Delta)^{\alpha/2}$. The latter operator satisfies a Harnack principle, whereas the operators that we consider here, do not satisfy the classical Harnack inequlity. In our framework one cannot expect to obtain bounds for the heat kernel that are rotationally symmetric as the bounds in \eqref{eq:aronson-bounds} and \eqref{eq:aronson-nonlocal}. Such a setting has been already studied for pure jump processes with singular jump kernels in \cite{Xu13} and \cite{KKK22}. The first article establishes sharp lower bounds and some upper bounds. The second articles invents some self-improving mechanism and proves sharp upper bounds. In the present work we study integro-differential operators resp. Markov processes in the Euclidean space $\R^{d+n}$, where the process performs jumps within the first $d$ coordinates and is a diffusion in the remaining $n$ coordinates. It turns out that the presence of jump and diffusive behavior at the same time is a challenge for proving sharp upper bounds of the heat kernel.

\medskip

Heat kernel estimates are closely linked to the parabolic Harnack inequality.  A parabolic version of the famous Harnack inequality goes back to Hadamard and Pini. It has been studied for many generators of Markov processes and the question whether resp. in which spaces the parabolic Harnack inequality and the Aronson bounds are equivalent has stimulated an interesting field of research. Note that the Harnack inequality is not essential for our work. It fails in its classical local form for jump processes with singular jump kernels, e.g., it fails for positive solutions to $\partial_t + (-\partial^2_1)^{\alpha/2} + (-\partial^2_2)^{\alpha/2} = 0$ in $B_1 \subset \R^2$. Hence, in general, it fails for the operators resp. the processes that we study, too. 

\subsection*{Main results} Let us explain the set-up and the main results. Let $d, n\in \N$. For $x\in \R^{d+n}$, $x=(x^1, \dots, x^{d+n})=x^1\e^1+\cdots +x^{d+n}\e^{d+n}$, where $\e^i$ is the unit vector, we define $x', \wt x\in \R^{d+n}$ by
\begin{align}\label{eq:not-tilde}
x':=\sum_{i=1}^d x\cdot \e^i,\quad \wt x:=\sum_{i=d+1}^{d+n} x\cdot\e^i.
\end{align}
Then, $x=x'+\wt x$ and for any $x,y\in\R^{d+n}$, $|x-y|^2=|x'-y'|^2+|\wt x-\wt y|^2$.
For $i\in \{1,\ldots,d\}$ let $\R_i:=\{z\in \R^{d+n}| z=a\e^i \;\mbox{for some}\;a\in\R \}$. For $x,y\in\R^{d+n}$, define
\begin{align*}
\scJ(x,y):=
\begin{cases}
\vspace{1mm}
\displaystyle\frac{1}{|x^i-y^i|^{1+\alpha}},\quad&\mbox{if}\;\;y-x\in \R_i\setminus\{0\} \text{ for } i\in\{1, \ldots,d\} \,,\\
\hskip 1cm 0& \mbox{otherwise}.
\end{cases}
\end{align*} 
Let $m(\rd y)$ be the product measure on $\bigcup_{i=1}^d \R_i$ such that $m$ restricted to each $\R_i$ is the one-dimensional Lebesgue measure on $\R$. Let $Z=(Z^1,\dots,  Z^{d+n})$ be a Markov process on $\R^{d+n}$ such that $Z^1, \dots,  Z^{d+n}$ are independent and $Z^i$ is a $1$-dimensional symmetric $\alpha$-stable process for all $i\in \{1, \ldots,d\}$ and $Z^j$ is $1$-dimensional Brownian motion for all $j\in \{d+1, \ldots,d+n\}$. Then, $Z$ is mixed singular symmetric Markov process in the sense that each component process $Z^i$ is either a diffusion or pure jump process. The Dirichlet form $(\scE, \scF)$ associated with $Z$ is given as follows:
\begin{align*}
\scE(u,v)&=\int\limits_{\R^{d+n}}\hskip-.04in\bigg(\int\limits_{\R^{d+n}}(u(y)-u(x))(v(y)-v(x))\scJ(x, y)m(\rd y)+\sum_{i=d+1}^{d+n}\partial_{i}u(x)\partial_{i}v(x)\bigg)\rd x,\\
\scF&=\{u\in L^2(\R^{d+n}):\scE(u,u)<\infty\}.
\end{align*}
Using the argument in \cite{Xu13}, one can check that $(\scE, \scF)$ is a regular Dirichlet form.
Let $\kappa\ge1$ and $J:\R^{d+n}\times \R^{d+n}\to (0,\infty)$ be a symmetric function satisfying
\begin{align}\label{J_comp}
{\kappa^{-1}}\scJ(x, y)\le J(x, y)\le {\kappa }\scJ(x, y).
\end{align}
Let $\scA(x):=(a_{ij}(x))_{1\le i,j\le d+n}$ be a measurable $(d+n)\times (d+n)$ matrix-valued function on $\R^{d+n}$ satisfying the following condition: 
\begin{itemize}
\item[(H)] For $i,j\in \{1, \ldots, d+n\} $,  $a_{ij}:\R^{d+n}\to \R$ is a measurable function such that 
\begin{align}
&a_{ij}(x)=a_{ji}(x),\quad \mbox{for almost all}\;x\in\R^{d+n},\label{aij_symm}
\end{align}
and
\begin{align}
\kappa^{-1}\big|\wt\xi \big|^2&\le \sum_{i,j=1}^{d+n} a_{ij}(x)\xi^i\xi^j\le \kappa\big|\wt\xi \big|^2,\quad \mbox{for all}\;x, \xi\in\R^{d+n} \,, \label{diff_ellip}
\end{align}
where we make use of the notation introduced in \eqref{eq:not-tilde}.
\end{itemize} 

\begin{remark*}
Note that condition (H) implies $a_{ij}\equiv 0$ if either $i$ or $j$ is in $\{1, \ldots, d\}$.
\end{remark*}

Define the symmetric Dirichlet form $(\cE, \cF)$ by
\begin{align}\label{def:DF}\begin{split}
\cE(u,v)&:=\int_{\R^{d+n}}\int_{\R^{d+n}}(u(y)-u(x))(v(y)-v(x))J(x, y)m(\rd y)\rd x\\
&\qquad +\int_{\R^{d+n}} \nabla u(x)\cdot \scA(x)\nabla v(x)\rd x,\\
\cF&:=\{u\in L^2(\R^{d+n}):\cE(u,u)<\infty\}.
\end{split}\end{align}
Then, by \eqref{J_comp}--\eqref{diff_ellip}, for any $u\in L^2(\R^{d+n})$
\begin{align}\label{form_equiv}
\kappa^{-1}\scE(u,u)\le \cE(u,u)\le \kappa\scE(u,u).
\end{align}
Thus, $\scF=\cF$ and $(\cE, \cF)$ is also a regular Dirichlet form. The Markov process associated with $(\cE, \cF)$ is denoted by $X=(X^1, X^2, \dots, X^{d+n})$. We use $(X_t)^i:=X_{t}^{i}:=X_t\cdot \e^i$ for $t>0$ and $i\in \{1,2,\dots, d+n\}$.

\begin{theorem}\label{theo:main}
Let $d,n\in\N$. Suppose $J$ satisfies \eqref{J_comp} and the functions $a_{ij}$ satisfy (H). Let $(\cE, \cF)$ be the Dirichlet form given by \eqref{def:DF}. Then, there is a conservative Hunt process $X=(X_t, \bP^x, x\in \R^{d+n}, t \ge 0)$ associated with $(\cE, \cF)$ that starts at every point in $\bR^{d+n}$. Moreover,  $X$ has a continuous transition density function $p(t,x,y)$ on $(0,\infty)\times\bR^{d+n}\times\bR^{d+n}$, with the following estimates: there exist  $c, C\ge1$ such that for any $(t,x,y)\in (0,\infty)\times \R^{d+n} \times \R^{d+n}$,
\begin{align}\begin{split}\label{eq:main-estim}
&C^{-1}t^{-d/\alpha-n/2}\prod_{i=1}^d \left(1\wedge \frac{t^{1/\alpha}}{|x^i-y^i|}\right)^{1+\alpha}\prod_{i = d+1}^{d+n} \exp\left(-\frac{c\,|x^i-y^i|^2}{t}\right)\\
&\le p(t,x,y)
\le C t^{-d/\alpha-n/2}\prod_{i=1}^d \left(1\wedge \frac{t^{1/\alpha}}{|x^i-y^i|}\right)^{1+\alpha}\prod_{i=d+1}^{d+n} \exp\left(-\frac{|x^i-y^i|^2}{ct}\right).
\end{split}\end{align}
\end{theorem}

\autoref{theo:main} proves pointwise robustness estimates for heat kernel estimates like the one of \cite{Aro68} for a diffusion (the case of $d=0$), the one of \cite{ChKu03} for isotropic jump processes and the one of \cite{Xu13, KKK22} for symmetric singular pure jump processes (the case of $n=0$).

\subsection*{Extensions} Let us discuss possible extensions and versions of our main result. The setting of \autoref{theo:main} seems to be restrictive because the first $d$ components of the stochastic process under consideration are of jump-type, whereas the last $n$ components form a non-degenerate diffusion. However, the analogous results holds true when considering a permutation $\sigma$ of the coordinates $\{1, \ldots, d+n\}$. One way to see this is by inspecting the proof. One could just replace $i,j, j_0 \in \{1, \ldots, d+n\}$ by $\sigma(i),\sigma(j), \sigma(j_0) \in \{1, \ldots, d+n\}$ in each step of the proof. But one could also establish the result by a formal consideration. To this end, one replaces the jump kernel $J$ and the coefficients $a_{ij}$ with the help of the perturbation by $J^\sigma$, $a_{ij}^\sigma$, e.g., $J^\sigma(x,y)=J(x^\sigma,y^\sigma)$. This leads to a new Dirichlet form $(\mathcal{E}^{\sigma}, D)$, where  $D =\{u\in L^2(\Rd)| \; \mathcal{E} (u,u)<\infty\}$ stays unchanged. We denote the Markov process corresponding to $(\mathcal{E}^\sigma, D)$ by $X^\sigma$, the corresponding semigroup by $P^\sigma$ and the corresponding heat kernel by $p^\sigma(t,x,y)$. The main observation now is
\begin{align}\label{eq:p-p_sigma}
p(t,x_0,y_0) = p^\sigma(t,x_0^\sigma,y_0^\sigma) \,
\end{align}
for all $t,x_0,y_0$.  Note that this property is not a mathematical triviality. Its proof is rather simple, though. \eqref{eq:p-p_sigma} follows once one has shown $P^\sigma_t f(x) = P_t (f \circ \sigma) (\sigma^{-1} (x))$ for every non-negative function $f$ and every $x \in \R^d$. This statements itself follows from the invariance of the Lebesgue measure with respect to permutations. A detailed discussion can be found at the end of the appendix in \cite{KKK22}.
 
\medskip

A second extension of \autoref{theo:main} concerns the type of jump process that is performed in the first $d$ coordinates.  The method of proof of \autoref{theo:main} allows us to obtain heat kernel estimates for different types of mixed processes. Consider the $d$-dimensional isotropic $\alpha$-stable process $Y$ and the $n$-dimensional Brownian motion $W$. Then, $\cZ=(Y,W)$ is a $d+n$-dimensional Markov process which is different from $Z$. Let $(\scE^{\mathrm{iso}}, \scF^{\mathrm{iso}})$ be the Dirichlet form associated with $\cZ$. Then,
\begin{align*}
\scE^{\mathrm{iso}}(u,v)&=\int_{\R^{d+n}}\int_{\R^{d+n}}(u(y)-u(x))(v(y)-v(x))\scJ^{\mathrm{iso}}(x, y)m_d(\rd y)\rd x\\
&\qquad+\int_{\R^{d+n}}\sum_{i\in \cA_2}\partial_{i}u(x)\partial_{i}v(x)\rd x,\\
\scF^{\mathrm{iso}}&=\{u\in L^2(\R^{d+n}):\scE^{\mathrm{iso}}(u,u)<\infty\},
\end{align*}
where
\begin{align*}
\scJ^{\mathrm{iso}}(x,y)=
\begin{cases}
\vspace{1mm}
\displaystyle\frac{1}{|x'-y'|^{d+\alpha}},\quad&\mbox{if}\;\;x'\neq y';\\
\hskip 1cm 0& \mbox{otherwise},
\end{cases}
\end{align*} 
and $m_d(\rd y)$ is the $d$-dimensional Lebesgue measure. Let $J^{\mathrm{iso}}:\R^{d+n}\times \R^{d+n}\to (0,\infty)$ be a symmetric function satisfying
\begin{align}\label{J_iso}
{\kappa^{-1}}\scJ^{\mathrm{iso}}(x, y)\le J^{\mathrm{iso}}(x, y)\le {\kappa }\scJ^{\mathrm{iso}}(x, y)
\end{align}
and define Dirichlet form
$(\cE^{\mathrm{iso}}, \cF^{\mathrm{iso}})$ by
\begin{align}\label{def:DF_iso}\begin{split}
\cE^{\mathrm{iso}}(u,v)&:=\int_{\R^{d+n}}\int_{\R^{d+n}}(u(y)-u(x))(v(y)-v(x))J^{\mathrm{iso}}(x, y)m_d(\rd y)\rd x\\
&\qquad +\int_{\R^{d+n}}\nabla u(x)\cdot \scA(x)\nabla v(x)\rd x,\\
\cF^{\mathrm{iso}}&:=\{u\in L^2(\R^{d+n}):\cE^{\mathrm{iso}}(u,u)<\infty\}.
\end{split}\end{align}
Then, by \eqref{J_iso}, \eqref{aij_symm} and \eqref{diff_ellip}, for any $u\in L^2(\R^{d+n})$
\begin{align*}
\kappa^{-1}\scE^{\mathrm{iso}}(u,u)\le \cE^{\mathrm{iso}}(u,u)\le \kappa\scE^{\mathrm{iso}}(u,u).
\end{align*}

\begin{theorem}\label{t:main2}
Let $d,n\in\N$. Suppose $J^{\mathrm{iso}}$ satisfies \eqref{J_iso} and the functions $a_{ij}$ satisfy (H). Let $(\cE^{\mathrm{iso}}, \cF^{\mathrm{iso}})$ be the Dirichlet form given by \eqref{def:DF_iso}. Then, there is a conservative Hunt process $\cX=(\cX_t, \bP^x, x\in \R^{d+n}, t \ge 0)$ associated with $(\cE^{\mathrm{iso}}, \cF^{\mathrm{iso}})$ that starts every point in $\bR^{d+n}$. Moreover,  $\cX$ has a continuous transition density function $p^{\mathrm{iso}}(t,x,y)$ on $(0,\infty)\times\bR^{d+n}\times\bR^{d+n}$, with the following estimates: there exist  $c, C\ge1$ such that for any $(t,x,y)\in (0,\infty)\times \R^{d+n} \times \R^{d+n}$,
\begin{align}\begin{split}\label{e:main2}
&C^{-1}t^{-d/\alpha-n/2}\left(1\wedge \frac{t^{1/\alpha}}{|x'-y'|}\right)^{d+\alpha}\exp\left(-\frac{c\,|\wt x-\wt y|^2}{t}\right)\\
&\le p^{\mathrm{iso}}(t,x,y)
\le C t^{-d/\alpha-n/2}\left(1\wedge \frac{t^{1/\alpha}}{|x'-y'|}\right)^{d+\alpha}\exp\left(-\frac{|\wt x-\wt y|^2}{ct}\right).
\end{split}\end{align}
\end{theorem}

\begin{proof}
The proof of this important case is analogous to the proof of \autoref{theo:main} in the case $d=n=1$. Then, instead of \eqref{eq:main-estim} one establishes \eqref{e:main2}.
\end{proof}

\subsection*{Notation} Let us comment the notation that we are using. As is usual, $\N_0$ denotes the non-negative integers including Zero. 
For two non-negative functions $f$ and $g$, the notation $f\asymp g$ means that there are positive constants $c_1$ and $c_2$ such that $c_1g(x)\leq f (x)\leq c_2 g(x)$ in the common domain of definition for $f$ and $g$. For $a, b\in \R$, we use $a\wedge b$ for $\min \{a, b\}$ and $a\vee b$ for $\max\{a, b\}$. Given any sequence $(a_n)$ of real numbers and $n_1, n_2\in \N_0$, we set $\prod_{n=n_1}^{n_2} a_n$ (resp. $\sum_{n=n_1}^{n_2} a_n$) as equal to $1$ (resp. $0$) if $n_1>n_2$. As explained above, for $x\in\R^{d+n}$ we write $x'=(x^1, \dots, x^d, 0, \dots, 0)$ and $\wt x=(0, \dots, 0, x^{d+1}, \dots, x^{d+n})$. 

\subsection*{Related results} Let us discuss related results about Markov processes in the Euclidean space from the literature. Note that there are many results about fine pointwise estimates for the transition density function $p(t,x,y)$ in the case of L\'{e}vy processes. Here, we do  not comment on those results because we focus on robustness results, i.e. we consider inhomogeneous settings. In particular, the equality  $p(t,x,y)=p(t,x-y,0)$ does not generally hold in the framework that we study.  

\medskip

As discussed in the beginning of the introduction, our work is inspired by the development initiated in \cite{Aro68} for diffusions and \cite{ChKu03} for a large class of symmetric Markov jump processes whose jump kernels are comparable to isotropic functions. 

\medskip

The articles \cite{CKK08, CKK11} prove heat kernel bounds for finite range jump processes and jump processes with exponentially decaying jump kernels, respectively. Sharp bounds in the case of polynomially decaying jump kernels are established in \cite{BKKL19}. Note that \cite{BKKL19} also allows for small jumps with an intensity that is stronger than for any $\alpha$-stable process. Pointwise heat kernel bounds are proved in \cite{CK10} for Markov processes generated by  Dirichlet forms that are given as the sum of local and nonlocal forms.
 
\medskip

The aforementioned cases cover a wide range of Markov processes. However, they all satisfy the property that the transition density function $p(t,x,y)$ is comparable to some given functions $g_{1/2}$, which would depend on the specific case, in the following way:
\begin{align*}
g_1(t,|x-y|) \leq p(t,x,y) \leq g_2(t,|x-y|) \,,
\end{align*} 
for any $t>0$ and all $x,y$ in the corresponding space. But there are many interesting cases where these rotational bounds do not hold true. A very simple case is given by the process $Z=(Z^1, Z^2)$, where $Z^1, Z^2$ are independent one-dimensional stable processes. The transitions density function of this process, which is the fundamental solution for the operator $\partial_t + (-\partial^2_1)^{\alpha/2} + (-\partial^2_2)^{\alpha/2}$, is the product of two solutions acting in one dimension, thus not rotational at all. It is very interesting to investigate robustness results for operators based on these examples. Such a program has been initiated in \cite{Xu13}, where $d$-dimensional jump processes with singular jump kernels are considered. \cite{Xu13} establishes sharp heat kernel lower bounds and rough off-diagonal upper bounds. Sharp off-diagonal upper bounds are proved in \cite{KKK22} with the help of some involved iterative scheme. Thus, \cite{KKK22} finally establishes a robustness result along the ideas of \cite{Aro68}, \cite{ChKu03} for Markov jump processes with singular jump kernels.  In the present work we use the scheme developed in \cite{KKK22} for those directions in $\R^{d+n}$, for which the process under consideration is governed by a jump process. Note that the scheme has also been used in \cite{KW22} to show sharp heat kernel bounds for some more general singular jump processes replacing the stable process in each direction by one fixed subordinate Brownian Motion. See also \cite{CHZ23}, \cite{KW23} for corresponding results on the Dirichlet heat kernel.

\medskip

All of the aforementioned results on jump processes have led to a conjecture about the robustness question for jump processes, which is discussed in the introduction of \cite{KKK22}. The present work shows that this conjecture seems to hold also for processes that satisfy comparability of the jump intensity for some coordinates and comparability with a given non-degenerate diffusion in the remaining coordinates.

\subsection*{Organization of the article} In \autoref{sec:NDE}, we obtain near-diagonal upper and lower bounds for the heat kernel. Moreover, we prove two important auxiliary results. The first one establishes rough upper bounds for the heat kernel, see \autoref{theo:uhk}. The second one is a survival estimate with respect to cubes in the corresponding metric, see \autoref{prop:survival}. In \autoref{sec:lbe} we prove the off-diagonal lower bounds in \autoref{theo:offdiag-lower}. \autoref{sec:suhk} is the heart of this work. First, in \autoref{sec:upper-dn-one}, we consider the special case $d=n=1$ and give a detailed, fully self-contained proof of the upper bound in \eqref{eq:main-estim}. \autoref{sec:upper-dn-general} and \autoref{s:proof} are devoted to the case of general $d,n\in\N$. In \autoref{sec:upper-dn-general} we explain the strategy of the proof in this case and, to this end, provide an iterative scheme. In \autoref{s:proof} we discuss auxiliary results and their proofs. 

\subsection*{Remark} This article was completed and uploaded to www.arxiv.org in 2021. Since the proof of near-diagonal lower bounds of $p(t,x,y)$ in \autoref{prop:ndl} requires Hölder regularity of the heat kernel, the authors decided to wait with the publication of the article until this result, see Theorem 1.6 and Theorem 7.13 of \cite[Version 2]{ChKaWe19}, is available on www.arxiv.org.

\subsection*{Acknowledgments} The authors of this article thank Takashi Kumagai for very helpful discussions on the subject of the article.

\section{Near-diagonal estimates and auxiliary results}\label{sec:NDE}

In this section we collect some important results, which can be established in a rather direct fashion. First, we prove on-diagonal upper bounds in \autoref{prop:uhkd}. They allow us to prove some useful, yet not-sharp, upper off-diagonal bounds in \autoref{theo:uhk}. Since our Markov process under consideration has direction-dependent behavior, we compensate this behavior by choosing a corresponding metric. With respect to this metric we can establish another important result, \autoref{prop:survival}, which is known as a survival estimate. It directly leads to the on-diagonal bounds from below in \autoref{prop:ndl}. The last result is concerned with the mean exit time. We show that the mean exit time w.r.t. appropriate cubes and balls behaves like the one of an isotropic process in a Euclidean ball, see \autoref{t:meanexit}.

\medskip

Recall that $X$ is the Markov process associated with $(\cE, \cF)$. 
We first introduce the L\'evy system for our stochastic processes with singular jump kernels. For the proof, see \cite[Appendix A]{ChKu08}.
\begin{lemma}
For any $x\in \R^{d+n}$, stopping time $S$ (with respect to the filtration of $X$), and non-negative measurable function $f$ on $\R_+ \times \R^{d+n}\times \R^{d+n}$ with $f(s, y, y)=0$ for all $y\in\R^{d+n}$ and $s\ge 0$, we have 
	\begin{equation}\label{eq:LS}
	\E^x \left[\sum_{s\le S} f(s,X_{s-}, X_s) \right] = \E^x \left[ \int_0^S \left(\sum_{i=1}^{d}\int_{\R} 
	f(s,X_s, X_s+e^i h) 
	J(X_s,X_s+e^i h) \rd h \right) \rd s \right].
	\end{equation}
\end{lemma}

\medskip

\begin{proposition}\label{prop:uhkd}
(i) There exists a constant $c>0$ such that for any $f\in \cF\cap L^1(\R^{d+n})$
\begin{align}\label{e:Nash}
\|f\|^{2+2(d/\alpha+n/2)^{-1}}_{2}\le c \,\cE(f,f)\|f\|_1^{2(d/\alpha+n/2)^{-1}}.
\end{align}
(ii) There is a properly exceptional set $\cN$ of $X$, a positive symmetric kernel $p(t,x,y)$ defined on $(0,\infty)\times (\bR^{d+n}\setminus\cN)\times(\bR^{d+n}\setminus\cN)$, and a constant $C>0$ such that $\bE^x[f(X_t)]= \int_{\R^{d+n}}p(t,x,y)f(y)\rd y$, and
\begin{align}\label{e:uhkd}
p(t,x,y)\le C t^{-d/\alpha-n/2}
\end{align}
for every $x,y\in\bR^{d+n}\setminus\cN$ and for every $t>0.$
\end{proposition}
\begin{proof}
Since $Z^1,\dots, Z^{d+n}$ are independent,  the heat kernel $q(t,x,y)$ for $Z$ satisfies
\begin{align*}
q(t,x,y)
\asymp  t^{-d/\alpha-n/2}\exp\left(-\frac{|\wt x-\wt y|^2}{4t}\right)\cdot\prod_{i=1}^{d}\left(1\wedge \frac{t^{1/\alpha}}{|x^i-y^i|}\right)^{1+\alpha}.
\end{align*}
Thus, by \cite[Theorem 2.1]{CKS87}, there exists $c_1>0$ such that for $f\in \cF\cap L^1(\R^{d+n})$, 
\begin{align*}
\|f\|^{2+2(d/\alpha+n/2)^{-1}}_{2}\le c_1 \,\scE(f,f)\|f\|_1^{2(d/\alpha+n/2)^{-1}}.
\end{align*}
Thus, \eqref{e:Nash} follows from \eqref{form_equiv}. 
Using \eqref{e:Nash}, \cite[Theorem 2.1]{CKS87} and \cite[Theorem 3.1]{BBCK09}, the second result follows.
\end{proof}

\begin{remark}\label{r:uhkd} In \cite{ChKaWe19}, the H\"older continuity for solutions of corresponding parabolic equations is established, see Theorem 1.6 and Theorem 7.13. Since $p(t,x,y)$ is a caloric function, i.e., it solves the corresponding heat-type equation, we may assume $\cN=\emptyset$, where $\cN$ is the properly exceptional set in \autoref{prop:uhkd}(ii). Thus, \eqref{e:uhkd} holds for all $x,y\in \R^{d+n}$ and $t>0$.
\end{remark}

\medskip

Next, we want to apply the Davies method in order to prove some off-diagonal upper bound. To this end, let $\rd\Gamma$ be the carr\'e du champ measure for $(\cE, \cF)$ . Then, $\cE(u,u)=\int_{\R^{d+n}}\rd \Gamma(u,u)$ and for $\psi\in\cF$,
\begin{align*}
\frac{\rd e^{-2\psi}\Gamma(e^{\psi},e^{\psi})}{\rd x}
&=\int_{\R^{d+n}}e^{-2\psi(x)}(e^{\psi(x)}-e^{\psi(y)})^2 J(x,y)m(\rd y) + e^{-2\psi(x)}\nabla e^{\psi(x)}\cdot \scA(x)\nabla e^{\psi(x)}\\
&=\int_{\R^{d+n}}(e^{\psi(y)-\psi(x)}-1)^2 J(x,y)m(\rd y) + \sum_{i,j\ge d+1}a_{ij}(x)\partial_i \psi(x)\partial_j \psi(x).
\end{align*} 
Define
\begin{align*}
\Gamma(f)(x)&:=\int_{\bR^{d+n}} (e^{f(y)-f(x)}-1)^2J(x,y)m(\rd y)+\sum_{i,j\ge d+1}a_{ij}(x)\partial_i f(x)\partial_j f(x),\\
\Lambda(f)^2&:=\|\Gamma(f)\|_{\infty}\vee\|\Gamma(-f)\|_{\infty},\\
E(t,x,y)&:=\sup\Big\{|f(x)-f(y)|-t\Lambda(f)^2: f\in \text{Lip}_c(\bR^d), \text{with}\;\Lambda(f)<\infty \Big\}.
\end{align*}
We use the Davies method to prove the following upper bound for the heat kernel. Although it is not sharp, it will play an important role in obtaining sharp upper bounds. 

\begin{theorem}\label{theo:uhk}
There exist $c, C \geq 1$ such that for all $t>0$ and $x,y\in\R^{d+n}$
\begin{align*}
p(t,x,y)\le C t^{-d/\alpha-n/2} \ \prod_{i=1}^{d}\bigg(1\wedge \frac{t^{1/\alpha}}{|x^i-y^i|}\bigg)^{\alpha/3} \ \exp\bigg(-\frac{|\wt x-\wt y|^2}{c\,t}\bigg)  .
\end{align*}
\end{theorem}

\begin{remark*}
In the proof we apply ideas of \cite{Xu13} to the first $d$ coordinates. Thus we arrive at the exponent $\alpha/3$ for these components, which is far from the optimal exponent. It is remarkable that, up to now, it seems unclear how to modify the technique by Carlen-Kusuoka-Stroock in order to prove optimal off-diagonal upper bounds. See the corresponding comments in \cite{KKK22}.
\end{remark*}

\begin{proof} 
Fix $x=(x^1, x^2, \dots, x^{d+n})$ and $y=(y^1, y^2, \dots, y^{d+n})$. For $i\in \{1,2,\dots, d+n\}$, let 
\begin{align*}
R_i&=|x^i-y^i|,\quad
\lambda_i=
\begin{cases}
(3R_i)^{-1}\log\big(({R_i^{\alpha}}/{t})\vee1\big),\quad& i\in \{1,2,\dots, d\};\\
{R_i}/{((\kappa+1) t)},\quad& i\in \{d+1, d+2,\dots, d+n\},
\end{cases}\\
\psi_i(\xi)&=\lambda_i(R_i-|\xi^i-x^i|)\vee 0,\qquad\text{for}\quad\xi=(\xi^1,\xi^2, \dots,\xi^{d+n})\in\R^{d+n},
\end{align*} 
and $\psi(\xi)=\sum_{i=1}^{d+n}\psi_i(\xi)$. Then, we see that for all $i\in \{1,2,\dots, d+n\}$ and $\xi, \zeta\in\R^{d+n}$,
\begin{align}\label{e:eq1}
|\psi_i(\xi)-\psi_i(\zeta)|\le \lambda_i|\xi^i-\zeta^i|, \quad |e^{\psi_i(\xi)-\psi_i(\zeta)}-1|\le 2e^{\lambda_i R_i}.
\end{align}
Moreover, for $i,j\in \{1,2,\dots, d+n\}$ with $i\neq j$,
\begin{align}\label{pd01}
\psi_i(\xi+\e^j s)=\lambda_i(R_i-|\xi^i+(\e^j s)^i-x^i|)\vee 0=\lambda_i(R_i-|\xi^i+0-x^i|)\vee 0=\psi_i(\xi).\quad
\end{align}
By \eqref{J_comp}, \eqref{diff_ellip} and \eqref{pd01},
\begin{align*}
\Gamma(\psi)(\xi)&\le \kappa\left(\sum_{i=1}^{d}\int_{\R}\bigg(e^{\psi(\xi+\e^i s)-\psi(\xi)}-1\bigg)^2\scJ(\xi, \xi+\e^i s)\rd s+\sum_{i=d+1}^{d+n}|\partial_{i}\psi(\xi)|^2\right)\nn\\
&=\kappa\left(\sum_{i=1}^{d}\int_{\R}\bigg(e^{\psi_i(\xi+\e^i s)-\psi_1(\xi)}-1\bigg)^2\scJ(\xi, \xi+\e^i s)\rd s+\sum_{i=d+1}^{d+n}|\partial_{i}\psi_i(\xi)|^2\right)\\
&=:\kappa \left(\sum_{i=1}^{d} I_i+ \sum_{i=d+1}^{d+n}II_i\right).
\end{align*}
For $i\in \{1,2,\dots, d\}$, if $R_i^{\alpha}\le t$, then $\psi_i\equiv0$, and thus,  $I_i\equiv0$. Now,  consider the case that $R_i^{\alpha}>t$ for $i\in \{1,2,\dots, d\}$. Using  \eqref{e:eq1} and  $(e^{s}-1)^2\le s^2e^{2|s|}$ for all $s\in \R$, we see that,
\begin{align*}
I_i&=\int_{\R}\big(e^{\psi_i(\xi+\e^i s)-\psi_i(\xi)}-1\big)^2\scJ(\xi, \xi+\e^i s)\rd s\\
&=c\int_{|s|\le R_i}\big(e^{\psi_i(\xi+\e^i s)-\psi_i(\xi)}-1\big)^2\frac{\rd s}{|s|^{1+\alpha}}+c\int_{|s|> R_i}\big(e^{\psi_i(\xi+\e^i s)-\psi_i(\xi)}-1\big)^2\frac{\rd s}{|s|^{1+\alpha}}\nn\\
&\le c\int_{|s|\le R_i}\frac{\lambda_i^2|s|^2e^{2\lambda_i |s|}}{|s|^{1+\alpha}}\rd s+ce^{2\lambda_i R_i}\int_{|s|> R_i}\frac{1}{|s|^{1+\alpha}}\rd s\nn\\
&\le c \lambda_i^2 e^{2\lambda_i R_i}R_i^{2-\alpha}+c e^{2\lambda_i R_i}R_i^{-\alpha}\nn\\
&\le c e^{3\lambda_i R_i}R_i^{-\alpha}
= c \frac{R_i^{\alpha}}{t}R_i^{-\alpha}=c t^{-1}.
\end{align*}
Also, using $|\psi_i(\xi)-\psi_i(\zeta)|\le \lambda_i|\xi^i-{\zeta}^i|$ again, we have $II_i\le  \lambda_i^2$. Thus, 
\begin{align*}
\Gamma({\psi})(\xi)\le \kappa\left(cdt^{-1}+ \sum_{i=d+1}^{d+n}\lambda_i^2\right).
\end{align*}
By definition, $\psi(x)-\psi(y)=\sum_{i=1}^{d+n}\big(\psi_i(x)-\psi_i(y)\big)=\sum_{i=1}^{d+n}\lambda_i R_i$.
Thus, by \autoref{prop:uhkd}, \autoref{r:uhkd} and \cite[Theorem 3.25]{CKS87},
\begin{align*}
p(t,x,y)&\le C t^{-d/\alpha-n/2}\exp\left(-\sum_{i=1}^{d+n}\lambda_i R_i+\kappa t\Big(cdt^{-1}+  \sum_{i=d+1}^{d+n}\lambda_i^2\Big)\right)\\
&= C t^{-d/\alpha-n/2}\exp\left(cd\kappa-\sum_{i=1}^{d}\lambda_i R_i- \sum_{i=d+1}^{d+n}\Big(\lambda_iR_i-\kappa t\lambda_i^2\Big)\right)\\
&\le C t^{-d/\alpha-n/2} \prod_{i=1}^{d}\bigg(1\wedge\frac{t}{|x^i-y^i|^{\alpha}}\bigg)^{1/3} \prod_{i=d+1}^{d+n}\exp\bigg(-\frac{|x^i-y^i|^2}{(\kappa+1)^2 t}\bigg).
\end{align*}
\end{proof}

The following lemma provides auxiliary computations.  

\begin{lemma}
Let $d, n, k\in\N$ and $c>0$.\\
(i) There exists a positive constant $c_1=c_1(d, n, c, \alpha)$ such that for any $r, t>0$
\begin{align}\label{eq:ineq1}
\int_{\{y\in\R^{n}:|y|\ge r^{\alpha/2}\}} t^{-n/2}\exp\bigg(-\frac{c|y|^2}{t}\bigg)\rd y\le c_1
\Big(\frac{t}{r^{\alpha}}\Big)^{1+d/\alpha}.
\end{align}
(ii) There exists a positive constant $c_2=c_2(n, c)$ such that for any $t>0$
\begin{align}\label{eq:ineq2}
\int_{\R^{n}} t^{-n/2}\exp\bigg(-\frac{c|y|^2}{t}\bigg)\rd y\le c_2.
\end{align}
(iii) There exists a positive constant $c_3=c_3(n, c, k)$ such that for $0<a<b\le ka$,
\begin{align}\label{eq:ineq3}
\int^{t}_{0}s^{-n/2}\int_{\{y\in\R^{n}:a<|y|<b\}} \exp\bigg(-\frac{c|y|^2}{s}\bigg)\rd y\rd s\le c_3t\exp\bigg(-\frac{ca^2}{2t}\bigg).
\end{align}
\end{lemma}
\begin{proof}
By using the spherical coordinate change and the change of variable $(s=\sqrt{t}u)$,
\begin{align*}
\int_{\{y\in\R^{n}:|y|\ge r^{\alpha/2}\}} & \Big(\frac{r^{\alpha}}{t}\Big)^{1+d/\alpha} t^{-n/2}\exp\bigg(-c\frac{|y|^2}{t}\bigg)\rd y\\
&=c(n)\int^{\infty}_{r^{\alpha/2}}\Big(\frac{r^{\alpha}}{t}\Big)^{1+d/\alpha} t^{-n/2}\exp\bigg(-c\frac{s^2}{t}\bigg)s^{n-1}\rd s \\
&=c(n)\int^{\infty}_{r^{\alpha/2}}\Big(\frac{r^{\alpha}}{t}\Big)^{1+d/\alpha} \Big(\frac{s^2}{t}\Big)^{n/2}\exp\bigg(-c\frac{s^2}{t}\bigg)s^{-1}\rd s \\
&\le c(n)\int^{\infty}_{r^{\alpha/2}} \Big(\frac{s^2}{t}\Big)^{1+d/\alpha+n/2}\exp\bigg(-c\frac{s^2}{t}\bigg)s^{-1}\rd s \\
&=c(n)\int^{\infty}_{(r^{\alpha}/t)^{1/2}}u^{1+2d/\alpha+n} \exp\big(-cu^2\big)\rd u\\
& \le c(n)\int^{\infty}_{0} u^{1+2d/\alpha+n}\exp\big(-cu^2\big)\rd u\le c_1, 
\end{align*}
which proves \eqref{eq:ineq1}. The proof of \eqref{eq:ineq2} using polar coordinates is standard. For \eqref{eq:ineq3}, using spherical coordinate change, we obtain
\begin{align*}
&s^{-n/2}\int_{\{y\in\R^{n}:a<|y|<b\}} \exp\bigg(-c\frac{|y|^2}{s}\bigg)\rd y
=c(n)s^{-n/2}\int_{a}^{b} \exp\bigg(-c\frac{u^2}{s}\bigg)u^{n-1}\rd u \\
&\le c(n)s^{-n/2}b^{n-1}\exp\bigg(-c\frac{a^2}{s}\bigg)(b-a)
\le  c(n)\left(\frac{b^2}{s}\right)^{n/2}\exp\bigg(-c\frac{a^2}{s}\bigg)\\
&\le  c(n)k^{n}\left(\frac{a^2}{s}\right)^{n/2}\exp\bigg(-c\frac{a^2}{s}\bigg).
\end{align*}
Since there exists $c'=c'(n,c)$ such that $u^{n/2}\exp(-cu)\le c'\exp(-cu/2)$ for all $u>0$,
\begin{align*}
&\int^{t}_{0}s^{-n/2}\int_{\{y\in\R^{n}:a<|y|<b\}} \exp\bigg(-c\frac{|y|^2}{s}\bigg)\rd y\rd s
\le c(n)k^{n}c'\int^{t}_{0}\exp\bigg(-\frac{c a^2}{2s}\bigg) \rd s\\
&\le c(n)k^{n}c'\exp\bigg(-\frac{c a^2}{2t}\bigg)\int^{t}_{0}\rd s= c_3t\exp\bigg(-\frac{c a^2}{2t}\bigg).
\end{align*}
\end{proof}

Next, let us introduce the truncated Dirichlet form $(\cE_\delta, \cF)$ together with its Markov process $X^{\delta}$  and the heat kernel $p^{\delta}(t,x,y)$.  For ${\delta}>0$, define $J_{\delta}(x,y):=J(x,y)\1_{\{|x'-y'|\le {\delta}\}}$
and for $u,v\in\cF$,
\begin{align*}
\cE_\delta(u,v)&:=\int_{\R^{d+n}}\int_{\R^{d+n}}(u(y)-u(x))(v(y)-v(x))J_\delta(x, y)m(\rd y) \rd x\\
&\qquad+\int_{\R^{d+n}}\nabla u(x)\cdot \scA(x)\nabla v(x)\rd x.
\end{align*}
Let $X^{\delta}$ be the Markov process associated with $(\cE_\delta, \cF)$ and $p^{\delta}(t,x,y)$ be the heat kernel for $X^{\delta}$. 
\begin{lemma}\label{l:J_outball}
There exists $C>0$ such that for any $x\in\R^{d+n}$,
\begin{align*}
C^{-1}\delta^{-\alpha}\le \int_{\R^{d+n}}\big(J(x,y)-J_{\delta}(x,y)\big)m(\rd y)\le C\delta^{-\alpha}.
\end{align*}
\end{lemma}
\begin{proof}
Using \eqref{J_comp},
\begin{align*}
&\int_{\R^{d+n}}\big(J(x,y)-J_{\delta}(x,y)\big)m(\rd y)=\int_{|y'-x'|>\delta}J(x,y)m(\rd y)\\
&\le \kappa\int_{|y'-x'|>\delta}\scJ(x,y)m(\rd y)= \kappa\sum_{i=1}^{d}\int_{|y^i-x^i|>\delta}\frac{1}{|x^i-y^i|^{1+\alpha}}\rd y^i=\frac{\kappa d}{\alpha}\delta^{-\alpha}.
\end{align*}
The lower bound can be proved analogously, see also \cite[Lemma 4.1]{Xu13}.
\end{proof}

\medskip

Using \autoref{l:J_outball},
\begin{align*}
\cE(u,u)-\cE_\delta(u,u)&=\int_{\R^{d+n}}\int_{\R^{d+n}}(u(x)-u(y))^2J(x,y)\1_{\{|x'-y'|>\delta\}}\,m(\rd y)\rd x\\
&\le 4\int_{\R^{d+n}}u(x)^2\,\rd x \sup_x\int_{|y'-x'|>\delta}J(x,y)\,m(\rd y)\\
&\le c\, \delta^{-\alpha}\|u\|_2^2.
\end{align*}
Thus, we obtain by \eqref{e:Nash} and \cite[Theorem 3.25]{CKS87} that
\begin{align}\label{p_trunc1}
p^{\delta}(t,x,y)\le c t^{-d/\alpha-n/2}e^{c_1 t \delta^{-\alpha}-E_\delta(2t,x,y)},
\end{align}
where
\begin{align*}
\Gamma_{\delta}(f)(x)&:=\int_{\bR^{d+n}} (e^{f(x)-f(y)}-1)^2J_{\delta}(x,y)m(\rd y)+\sum_{i,j\ge d+1}a_{ij}(x)\partial_i f(x)\partial_j f(x),\\
\Lambda_{\delta}(f)^2&:=\|\Gamma_{\delta}(f)\|_{\infty}\vee\|\Gamma_{\delta}(-f)\|_{\infty},\\
E_{\delta}(t,x,y)&:=\sup\Big\{|f(x)-f(y)|-t\Lambda_{\delta}(f)^2: f\in \text{Lip}_c(\bR^d), \text{with}\;\Lambda_{\delta}(f)<\infty\Big \}.
\end{align*}

\medskip

The following definition of a square takes into account the direction-dependent behavior of our process. For $x\in\R^{d+n}$, we define a modified cube in $\R^{d+n}$ with ``radius'' $r$ by
\begin{align}\label{cube}
\begin{split}
\cQ(x,r) :=\{y\in\R^{d+n}:|x^i-y^i|< r, |x^j-y^j|< r^{\alpha/2} \text{ for } 1\le i\le d < j\le d+n\}.
\end{split}
\end{align}

We also define a $r$-neighborhood of $x$ in $\R^{d+n}$ by
\begin{align}\label{nbd}
\cB(x,r):=\{y\in\R^{d+n}:|x'-y'|< r, |\wt x-\wt y|< r^{\alpha/2}\}.
\end{align}

Then, $\cB(x,r)\subset \cQ(x,r)$ holds for all $x\in\R^{d+n}$ and $r>0$. Moreover, there exists $c=c(d, n, \alpha)$ such that $\cQ(x, cr)\subset \cB(x,r)$ holds for all $x\in\R^{d+n}$ and $r>0$. Indeed, for $c=d^{-1/2}\wedge n^{-1/\alpha}$, let $y\in \cQ(x, cr)$. Then, $|x^i-y^i|^2<(cr)^2$ for $i\in \{1,2,\dots, d\}$ and $|x^j-y^j|^2<(cr)^{\alpha}$ for $j\in \{d+1,d+2,\dots, d+n\}$. This implies that $|x'-y'|^2=\sum_{i=1}^{d}|x^i-y^i|^2<d(cr)^2\le r^2$ and $|\wt x-\wt y|^2=\sum_{j=d+1}^{d+n}|x^j-y^j|^2<n(cr)^{\alpha}\le r^{\alpha}$.
Thus,
\begin{align}\label{rel_cube_nbd}
\cQ(x, cr)\subset \cB(x,r)\subset \cQ(x,r).
\end{align}

The definition of $\cQ(x,r)$ allows us prove a survival estimate despite the fact that the process is highly anisotropic. 

\begin{proposition}\label{prop:survival}
There exists a constant $C>0$ such that 
\begin{align*}
\bP^x(\tau_{\cQ(x,r)}\le t)\le C tr^{-\alpha}
\end{align*}
for all $t, r>0$ and $x\in\R^{d+n}$. 
\end{proposition}
\begin{proof}
Suppose $2t< r^\alpha$. 
Fix $x\in \R^{d+n}$ and $y\in \R^{d+n}\setminus\cB(x,r).$
Let ${\delta}=\frac{r \alpha}{3(d+\alpha)}$ and
\begin{align*}
R_1&=|x'-y'|,\quad \lambda_1=(3{\delta})^{-1}\log({{\delta}^{\alpha}}/{t}),\\
R_2&=|\wt x-\wt y|,\quad\lambda_2={R_2}/{((\kappa n+ 1)t)}.
\end{align*} 
Define
$$\psi'(\xi):=\lambda_1(R_1-|\xi'-x'|)\vee 0,\;\; \wt\psi(\xi):=\lambda_2(R_2-|\wt \xi-\wt x|)\vee 0,\quad \text{for}\;\;\xi\in\R^{d+n},$$
and $\psi(\xi):=\psi'(\xi)+\wt\psi(\xi)$.  Then, we observe that for $1\le i\le d$ and $d+1\le j\le d+n$,
\begin{align}\label{pd02}
\begin{split}
\psi'(\xi+\e^j s)&=\lambda_1(R_1-|(\xi+\e^j s)'-x'|)\vee 0=\lambda_1(R_1-|\xi'-x'|)\vee 0=\psi'(\xi),\\
\wt\psi(\xi+\e^i s)&=\lambda_2(R_2-|\wt{(\xi+\e^i s)}-\wt x|)\vee 0=\lambda_2(R_2-|\wt\xi-\wt x|)\vee 0=\wt\psi(\xi).
\end{split}
\end{align}
Using this, \eqref{J_comp} and \eqref{diff_ellip}, 
\begin{align*}
\Gamma_\delta({\psi})(\xi)&=\sum_{i=1}^{d}\int_{\R}\Big(e^{\psi(\xi+\e^i s)-\psi(\xi)}-1\Big)^2J_\delta(\xi, \xi+\e^i s)\rd s+\sum_{i=d+1}^{d+n}a_{ij}(x)\partial_i \psi(x)\partial_j \psi(x)\nn\\
&=\kappa\left(\sum_{i=1}^{d}\int_{|s|\le\delta}\Big(e^{\psi'(\xi+\e^i s)-\psi'(\xi)}-1\Big)^2\scJ(\xi, \xi+\e^i s)\rd s+\sum_{i=d+1}^{d+n}|\partial_{i}\wt\psi(\xi)|^2\right)\nn\\
&=:\kappa \sum_{i=1}^{d}I_i'+ \kappa\sum_{i=d+1}^{d+n} II_i.
\end{align*}
For $\xi, \zeta\in\R^{d+n}$, we have $|\psi'(\xi)-\psi'(\zeta)|\le \lambda_1|\xi'-\zeta'|$.
Using this and $(e^{s}-1)^2\le s^2e^{2|s|}$ for all $s\in \R$, we see that for $i\in \{1,2,\dots, d\}$
\begin{align*}
I_i'=\int_{|s|\le {\delta}}\Big(e^{\psi'(\xi+\e^i s)-\psi'(\xi)}-1\Big)^2\scJ(\xi, \xi+\e^i s)\rd s
\le\int_{|s|\le {\delta}}\frac{\lambda_1^2|s|^2e^{2\lambda_1 |s|}}{|s|^{1+\alpha}}\rd s
\le c e^{3\lambda_1 {\delta}}{\delta}^{-\alpha}.
\end{align*}
Also, using $|\wt\psi(\xi)-\wt\psi(\zeta)|\le \lambda_2|\wt\xi-\wt\zeta|$, we have $II_i\le  \lambda_2^2$ for $i\in \{d+1,d+2,\dots, d+n\}$. Thus, 
\begin{align*}
\Gamma_\delta(\psi)(\xi)\le \kappa\left(cde^{3\lambda_1 {\delta}}{\delta}^{-\alpha}+n\lambda_2^2\right),
\end{align*}
and
\begin{align*}
-E_{\delta}(2t,x,y)&\le -\lambda_1R_1-\lambda_2R_2+ \frac{\kappa cdt}{\delta^{\alpha}}e^{3\lambda_1 {\delta}}+\kappa nt\lambda_2^2\\
&=\kappa cd-\lambda_1R_1-\frac{R_2^2}{(\kappa+n)^2t}.
\end{align*}
Thus, by \eqref{p_trunc1} and $\delta^{\alpha}>2t$,
\begin{align}\label{trunc_upper}
p^{{\delta}}(t,x,y)&\le C t^{-d/\alpha-n/2}\exp(c t {\delta}^{-\alpha}-E_{\delta}(2t, x,y))\nn\\
&\le C t^{-d/\alpha-n/2}\exp\Big(c-\lambda_1R_1-\frac{R_2^2}{(\kappa n+1)^2t}\Big)\nn\\
&= C' t^{-d/\alpha}\exp\big(-\lambda_1|x'-y'|\big)\,t^{-n/2}\exp\Big(-\frac{|\wt x-\wt y|^2}{ct}\Big).
\end{align}
Let
\begin{align*}
 E_1&:=E_1(x,r):=\{y\in \R^{d+n}: |x'-y'|\ge r\},\\
 E_2&:=E_2(x,r):=\{y\in \R^{d+n}: |x'-y'|< r, \;|\wt x-\wt y|\ge r^{\alpha/2} \},
\end{align*} 
so that $\cB(x,r)^c=E_1\cup E_2$. Then,
\begin{align}\label{trunc_surv}
\bP^{x}\Big( X^{\delta}_t\in \cB(x,r)^c \Big)
&=\int_{y\in\cB(x,r)^c}p^{\delta}(t,x,y)\rd y\nn\\
&=\bigg(\int_{y\in E_1}+\int_{y\in E_2}\bigg) \,p^{\delta}(t,x,y)\rd y
=:\bI_1+\bI_2.
\end{align}
By \eqref{trunc_upper}, \eqref{eq:ineq2}, integration by parts and $\log(\delta^{\alpha}/t)>\log 2$,
\begin{align}\label{I1}
\bI_1&\le c t^{-d/\alpha}\int_{|x'-y'|\ge r}\exp\big(-\lambda_1|x'-y'|\big)\rd y'\int_{\R^{n}} t^{-n/2}\exp\Big(-\frac{|\wt x-\wt y|^2}{c\,t}\Big)\rd \wt y \nn\\
&\le c t^{-d/\alpha}\int^{\infty}_{r}\exp\big(-\lambda_1u\big)u^{d-1}\rd u\le ct^{-d/\alpha}\sum^{d}_{k=1}\frac{1}{\lambda_1^k}\exp(-\lambda_1 r)r^{d-k}\nn\\
&\le ct^{-d/\alpha}\sum^{d}_{k=1}\Big(\frac{3\delta}{\log 2}\Big)^k\Big(\frac{t}{\delta^{\alpha}}\Big)^{1+d/\alpha}r^{d-k}\le ct^{-d/\alpha}\sum^{d}_{k=1}r^k\Big(\frac{t}{r^{\alpha}}\Big)^{1+d/\alpha}r^{d-k}\nn\\
&\le ctr^{-\alpha}.
\end{align}
Also, by \eqref{trunc_upper} and \eqref{eq:ineq1}
\begin{align}\label{I2}
\bI_2&\le c \bigg(\int_{|x'-y'|< r}t^{-d/\alpha}\rd y'\bigg)\bigg(\int_{|\wt x-\wt y|\ge r^{\alpha/2}} t^{-n/2}\exp\Big(-\frac{|\wt x-\wt y|^2}{c\,t}\Big)\rd \wt y \bigg)\nn\\
&\le c t^{-d/\alpha}r^{d} t^{1+d/\alpha}r^{-\alpha-d}\nn\\
&= c tr^{-\alpha}.
\end{align}
Thus, by \eqref{rel_cube_nbd}, \eqref{trunc_surv}, \eqref{I1} and \eqref{I2}, for any $2t<r^{\alpha}$
\begin{align}\label{trunc_surv2}
\bP^{x}\big( X^{\delta}_t\in \cQ(x,r)^c \big)\le \bP^{x}\big( X^{\delta}_t\in \cB(x,r)^c \big)\le c tr^{-\alpha}.
\end{align}
Choose $c_1>2^{2/\alpha}$ so that $2(r/c_1)^{\alpha/2}<r^{\alpha/2}$ and $2(r/c_1)<r$. Then,
for any $z\in \cQ(x,r)^c$, $\cQ(z,r/c_1)\cap\cQ(x,r/c_1)=\emptyset$.
Thus, by the strong Markov property and \eqref{trunc_surv2}, for any $4t<(r/c_1)^{\alpha}$
\begin{align*}
\bP^{x}(\tau^{\delta}_{\cQ(x,r)}\le t)
&= \bP^{x}\big(\tau^{\delta}_{\cQ(x,r)}\le t, X^{\delta}_{2t}\in \cQ(x,r/c_1)^c\big)+\bP^{x}\big(\tau^{\delta}_{\cQ(x,r)}\le t, X^{\delta}_{2t}\in \cQ(x,r/c_1)\big)\\
&\le \bP^{x}\big(X^{\delta}_{2t}\in \cQ(x,r/c_1)^c\big)+\sup_{z\in \cQ(x,r)^c, s\le t}\bP^{z}\big(X^{\delta}_{2t-s}\in \cQ(z,r/c_1)^c\big)\\
&\le \bP^{x}\big(X^{\delta}_{2t}\in \cQ(x,r/c_1)^c\big)+\sup_{s\le t}\bP^{z}\big(X^{\delta}_{2t-s}\in \cQ(z,r/c_1)^c\big)\\
&\le c t r^{-\alpha}.
\end{align*}

From Meyer's construction and \autoref{l:J_outball},
\begin{align*}
\bP^{x}(X_s\neq X_s^{\delta}\;\text{for some}\;s\le t)\le t\sup_{z}\int_{\R^{d+n}} |J(z,y)-J_{\delta}(z,y)|m(\rd y)
\le c t r^{-\alpha}.
\end{align*}
Thus, for $4t<(r/c_1)^{\alpha}$,
\begin{align*}
\bP^{x}(\tau_{\cQ(x,r)}\le t)&\le \bP^{x}\big(\tau^{\delta}_{\cQ(x,r)}\le t\big)+\bP^{x}(X_s\neq X_s^{\delta}\;\text{for some}\;s\le t)\le c t r^{-\alpha}.
\end{align*}
For $4t\ge(r/c_1)^{\alpha}$, the result is trivial.
\end{proof}

\medskip

Finally, we can establish the desired on-diagonal lower bound. 

\begin{proposition}\label{prop:ndl}
There exist constants $c>0$ and $\eps>0$ such that
\begin{align*}
p(t,x,y)\ge c t^{-d/\alpha-n/2}\qquad\text{for}\;\;\;y\in \cQ(x, \eps t^{1/\alpha}).
\end{align*}
\end{proposition}
\begin{proof}
By \autoref{prop:survival} and \eqref{rel_cube_nbd}, there exists $c_1>0$ such that
$$\bP^x(\tau_{\cQ(x,r)}<t)\le c_1 tr^{-\alpha}.$$
Using this, we see that
\begin{align*}
\int_{\R^{d+n}\setminus \cQ(x, (c_1 t)^{1/\alpha})}p(t/2, x,y)\rd y \le \bP^x(\tau_{\cQ(x, (c_1 t)^{1/\alpha})}< t/2)\le \frac12.
\end{align*}
Thus, by Jensen's inequality,
\begin{align}\label{odl}
p(t,x,x)&=\int_{\R^{d+n}}p(t/2, x,y)^2\rd y
\ge \int_{\cQ(x, (c_1 t)^{1/\alpha})}p(t/2, x,y)^2\rd y\nn\\
&\ge \frac{1}{|\cQ(x, (c_1 t)^{1/\alpha})|}\bigg(\int_{\cQ(x, (c_1 t)^{1/\alpha})}p(t/2, x,y)\rd y\bigg)^2
\ge c_2 t^{-d/\alpha-n/2}.
\end{align}
Note that $c_2$ is  independent of $t>0$ and $x\in \R^{d+n}$. On the other hand, by the H\"older continuity for $p(t,x,\cdot)$ proved in Theorem 1.6 and Theorem 7.13 of \cite{ChKaWe19}, we can take $\eps=\eps(c_2)$ such that
\begin{align*}
|p(t,x,y)-p(t,x,z)|\le \frac{c_2}{2} t^{-d/\alpha-n/2} \quad\text{for all}\;\;y,z\in \cQ(x, \eps t^{1/\alpha}).
\end{align*}
Thus, by \eqref{odl} and the above inequality for $y\in \cQ(x, \eps t^{1/\alpha})$,
\begin{align*}
p(t,x,y)\ge p(t,x,x)-\frac{c_2}{2}t^{-d/\alpha-n/2}\ge \frac{c_2}{2} t^{-d/\alpha-n/2}.
\end{align*}
\end{proof}

\begin{proposition}\label{conservative}
The process $X$ is conservative; that is, $X$ has infinite lifetime.
\end{proposition}
\begin{proof}
Since the Dirichlet form $(\cE, \cF)$ admits no killings inside $\R^{d+n}$, the result follows from
\autoref{prop:ndl} and \cite[Proposition 3.1(2)]{CKW21}.
\end{proof}

\begin{theorem}\label{t:meanexit}
(i) There exists a constant $c_1>0$ such that for $x_0\in \R^{d+n}$ and $r>0$,
\begin{align*}
\bE^x[\tau_{\cQ(x_0,r)}]\le c_1 r^{\alpha}
\end{align*}
for all $x\in \cQ(x_0,r)$.\\
(ii) There exists a constant $c_2>0$ such that for $r>0$,
\begin{align*}
\bE^x[\tau_{\cQ(x,r)}]\ge c_2 r^{\alpha}
\end{align*}
\end{theorem}
\begin{proof}
(i) Let $C>0$ be the constant in \autoref{prop:uhkd}(ii). Take large $c_3$ so that $2C \le c_3^{d/\alpha+n/2}$. Then, for every $r>0$, $x_0\in\R^{d+n}$ and $x\in B(x_0,r)$, with $t:=c_3 r^{\alpha}$, we have by \autoref{prop:uhkd}(ii) and \autoref{r:uhkd}
\begin{align*}
\bP^x(X_t\in \cQ(x_0, r))=\int_{\cQ(x_0, r)} p(t,x,y)\rd y \le \frac{C |\cQ(x_0, r)|}{t^{d/\alpha+n/2}}= \frac{Cr^{d+n\alpha/2}}{c_3^{d/\alpha+n/2}r^{d+n\alpha/2}}\le \frac12.
\end{align*}  
Since $X$ is  conservative, it follows that for every $x\in \cQ(x_0, r)$,
\begin{align*}
\bP^x(\tau_{\cQ(x_0, r)}\le t)\ge\bP^x(X_t\notin \cQ(x_0, r))\ge\frac12,
\end{align*}  
which implies $\bP^x(\tau_{\cQ(x_0, r)}> t)\le 1/2.$ By the strong Markov property, for integer $k\ge1$,
\begin{align*}
\bP^x\big(\tau_{\cQ(x_0, r)}> (k+1)t\big)\le\bE^x\Big[\bP^{X_{kt}}(\tau_{\cQ(x_0, r)}> t);\tau_{\cQ(x_0, r)}>kt\Big]\le\frac12\bP^{x}(\tau_{\cQ(x_0, r)}> kt).
\end{align*}  
Using induction, we obtain that for every $k\ge1$,
\begin{align*}
\bP^x(\tau_{\cQ(x_0, r)}> kt)\le 2^{-k},
\end{align*}  
which implies
\begin{align*}
\bE^x[\tau_{\cQ(x_0, r)}]\le \sum^{\infty}_{k=0}t(k+1)\bP^x(\tau_{\cQ(x_0, r)}> kt)\le c_4 r^{\alpha}.
\end{align*}  

\noindent(ii) Let $C>0$ be the constant in \autoref{prop:survival} and  $t:={ r^{\alpha}}/{(2C)}$. By  \autoref{prop:survival},
\begin{align*}
\bE^x[\tau_{\cQ(x_0, r)}]\ge t\;\bP^x(\tau_{\cQ(x_0, r)}\ge t)
=t\big(1-\bP^x(\tau_{\cQ(x_0, r)}< t)\big)\ge t(1-1/2)=\frac{r^{\alpha}}{4C}.
\end{align*}
\end{proof}

\section{Off-diagonal lower bound}\label{sec:lbe}

In this section, we will prove sharp off-diagonal lower bounds for the heat kernel. The following lemma is a key ingredient for the lower bound estimate. 

\begin{lemma}\label{l:surv_lb}
There exist $C_1, C_2>0$ such that for all $x\in \R^{d+n}$ and $r>0$,
\begin{align*}
\bP^x(\tau_{\cQ(x,r)}>t)=\bP^x\Big(\sup_{s\le t} \big[ \bigvee_{i=1}^{d}|X_s^i-x^i|\vee \bigvee_{i=d+1}^{d+n}|X_s^i-x^i|^{2/\alpha} \big] \le r\Big)\ge C_1e^{-C_2tr^{-\alpha}}.
\end{align*}
\end{lemma}

\begin{proof}
Fix $x\in \R^{d+n}$ and $r>0$. Let $a\in(0,1)$ be a constant which will be chosen later. Let $t_0=ar^{\alpha}$
and 
$$\scD:=\left\{X_{t_0}\in \cQ(x, r/3),\;\; \sup_{s\le t_0}\bigvee_{i=1}^{d}|X_s^i-x^i|\vee\bigvee_{i=d+1}^{d+n}|X_s^i-x^i|^{2/\alpha}\le 2r/3\right\},$$
where $X_s=(X_s^1, \dots,  X_s^{d+n})$. For $y\in \cQ(x,r/3)$, by \autoref{prop:ndl},
\begin{align*}
&\bP^y(X_{t_0}\in \cQ(x, r/3))=\int_{\cQ(x, r/3)}p(t_0, y,z)\rd z\ge \int_{\cQ(x, r/3)\cap \cQ(y,\eps t_0^{1/\alpha})}p(t_0, y,z)\rd z\nn\\
&\ge c_3(t_0)^{-d/\alpha-n/2}|\cQ(x, r/3)\cap \cQ(y, \eps(t_0)^{1/\alpha})|
\ge c_4(t_0)^{-d/\alpha-n/2}(\eps(t_0)^{1/\alpha})^{d+n\alpha/2}\nn\\
&=c_4\eps^{d+n\alpha/2},
\end{align*}
where the last inequality holds since $r=a^{-1/\alpha}t_0^{1/\alpha}\ge t_0^{1/\alpha}$.

On the other hand, for $c_5:=(2^{\alpha/2}-1)^{2/\alpha}/3$ and $y\in \cQ(x,r/3)$, we observe that
$\cQ(y, c_5r)\subset \cQ(x, 2r/3)$. Using this  and \autoref{prop:survival},
\begin{align*}
&\bP^y\left(\sup_{s\le t_0}\bigvee_{i=1}^{d}|X_s^i-x^i|\vee \bigvee_{i=d+1}^{d+n}|X_s^i-x^i|^{2/\alpha}> 2r/3\right)\\
&\le \bP^y(\tau_{\cQ(x, 2r/3)}\le t_0)\le \bP^y(\tau_{\cQ(y, c_5 r)}\le t_0)\le c_6 a,
\end{align*}
where $c_6$ is independent of $a$. Thus, we obtain that for any $y\in\cQ(x,r/3)$,
\begin{align*}
\bP^y(\scD)&\ge \bP^y(X_{t_0}\in \cQ(x, r/3))-\bP^y\left(\sup_{s\le t_0}\bigvee_{i=1}^{d}|X_s^i-x^i|\vee \bigvee_{i=d+1}^{d+n}|X_s^i-x^i|^{2/\alpha}> 2r/3\right)\\
&\ge c_4\eps^{d+n\alpha/2}-c_6a.
\end{align*}

 Choosing $a$ sufficiently small, we obtain that there exists $b\in(0,1)$ such that
$$\inf_{y\in\cQ(x,r/3)}\bP^y(\scD)\ge b.$$
Using this estimate and the Markov property, 
\begin{align*}
&\bP^x\Big(\sup_{s\le 2t_0}\bigvee_{i=1}^{d}|X_s^i-x^i|\vee \bigvee_{i=d+1}^{d+n}|X_s^i-x^i|^{2/\alpha}\le r\Big)\nn\\
&\ge\bP^x\Big(X_{t_0}\in \cQ(x, r/3),\;\;\sup_{s\le t_0}\bigvee_{i=1}^{d}|X_s^i-x^i|\vee \bigvee_{i=d+1}^{d+n}|X_s^i-x^i|^{2/\alpha}\le 2r/3,\nn\\ &\qquad\qquad X_{2t_0}\in \cQ(x, r/3),\;\;\sup_{t_0\le s\le 2t_0}\bigvee_{i=1}^{d}|X_s^i-x^i|\vee \bigvee_{i=d+1}^{d+n}|X_s^i-x^i|^{2/\alpha}\le 2r/3\Big)\nn\\
&\ge \bP^x(\bP^{X_{t_0}}(\scD), \scD)\nn\\
&\ge b^2.
\end{align*}
By induction, we get for $k\in\N$,
\begin{align*}
\bP^x(\tau_{\cQ(x,r)}>kt_0)=\bP^x\Big(\sup_{s\le kt_0}\bigvee_{i=1}^{d}|X_s^i-x^i|\vee \bigvee_{i=d+1}^{d+n}|X_s^i-x^i|^{2/\alpha}\le r\Big)\ge b^k.
\end{align*}
Now, for $t>0$, choose $k$ to be the smallest integer greater than $t/t_0$. Then,
\begin{align*}
\bP^x(\tau_{\cQ(x,r)}>t)\ge\bP^x(\tau_{\cQ(x,r)}>kt_0)\ge b^k\ge b^{t/t_0+1}=b\exp\Big(\frac{t}{t_0}\log b\Big)=b\exp\Big(\frac{\log b}{a}tr^{-\alpha}\Big).
\end{align*}
Thus, taking $c_1=b$ and $c_2=-\frac{\log b}{a}>0$, we obtain the result.
\end{proof}

\begin{theorem}\label{theo:offdiag-lower}
There exist constants $c, C \geq 1$ such that for all $t>0$ and $x,y \in \R^{d+n}$
\begin{align*}
p(t,x,y)\ge C\, t^{-d/\alpha-n/2}\exp\bigg(-\frac{|\wt x-\wt y|^2}{c \, t}\bigg) \prod_{i=1}^{d}\bigg(1\wedge \frac{t^{1/\alpha}}{|x^i-y^i|}\bigg)^{1+\alpha}.
\end{align*}
\end{theorem}
\begin{proof}
Fix $x=(x^1, \dots, x^{d+n})=(x', \wt x)$ and $y=(y^1, \dots,  y^{d+n})=(y', \wt y)$. By \autoref{prop:ndl}, we may and do assume that $|x^i-y^i|>\eps t^{1/\alpha}$ for $i\in \{1,\dots, d\}$ and $r_0:=|\wt x-\wt y|>\eps^{\alpha/2}t^{1/2}$, where $\eps\in(0,1)$ is the constant in \autoref{prop:ndl}. 
Let $k\in\N$ be the smallest integer satisfying $r_0/k<3^{-1}\eps^{\alpha/2}(2^{-1}t/k)^{1/2}$. Then, $k\asymp r_0^2/t$. Indeed, $1\le 2\cdot3^2\eps^{-\alpha}{r_0^2}/{t}\le k <4\cdot3^2\eps^{-\alpha}{r_0^2}/{t}$ and thus
\begin{align}\label{k_comp}
36^{-1}\eps^{\alpha}\frac{t}{k}\le \Big(\frac{r_0}{k}\Big)^2 <18^{-1}\eps^{\alpha}\frac{t}{k}.
\end{align}
For $l=0,1,\dots, k$, let $z_l:=(x', \wt x+\frac{l}{k}(\wt y-\wt x))$ and $Q_l:=\cQ(z_l, (r_0/k)^{2/\alpha})$. For $l=0,1, \dots, k-1$ and $ i\in\{1,\dots, d\}$, $j\in\{d+1,\dots, d+n\}$, we have
$$|z_l^i-z_{l+1}^i|=0,\quad |z_l^j-z_{l+1}^j|\le |z_l-z_{l+1}|=\frac{r_0}{k}<3^{-1}\eps^{\alpha/2}\left(\frac{t}{2k}\right)^{1/2}.$$ 
Thus, for any ${\xi}_l\in Q_l$, ${\xi}_{l+1}\in Q_{l+1}$ and $ i\in\{1,\dots, d\}$, $j\in\{d+1,\dots, d+n\}$,
\begin{align*}
|{\xi}_l^i-{\xi}_{l+1}^i|&\le |{\xi}_l^i-z_l^i|+|z_l^i-z_{l+1}^i|+|z_{l+1}^i-{\xi}_{l+1}^i|< 2\left(\frac{r_0}{k}\right)^{2/\alpha}< \eps \left(\frac{t}{2k}\right)^{1/\alpha},\nn\\
|{\xi}_l^j-{\xi}_{l+1}^j|&\le |{\xi}_l^j-z_l^j|+|z_l^j-z_{l+1}^j|+|z_{l+1}^j-{\xi}_{l+1}^j|< 3\frac{r_0}{k}< \eps^{\alpha/2} \left(\frac{t}{2k}\right)^{1/2}.
\end{align*}
Thus, by  \autoref{prop:ndl}, we have $p(\frac{t}{2k},{\xi}_l,{\xi}_{l+1})\ge c_1(2^{-1}t/k)^{-d/\alpha-n/2}$. Then, using the semigroup property and the relation \eqref{k_comp}, we obtain
\begin{align}\label{lower1}
&p(t,x,y)\nn\\
&=\int_{\R^{d+n}}\int_{\R^{d+n}}\cdots\int_{\R^{d+n}}p(\frac{t}{2k}, x, {\xi}_1)p(\frac{t}{2k}, {\xi}_1, {\xi}_2)\cdots p(\frac{t}{2k}, {\xi}_{k-1}, \xi_k) p(\frac{t}{2}, {\xi}_{k}, y)\rd {\xi}_1\cdots \rd {\xi}_{k}\nn\\
&\ge\int_{Q_1}\int_{Q_2}\cdots \int_{Q_k} p(\frac{t}{2k}, x, {\xi}_1)p(\frac{t}{2k}, {\xi}_1, {\xi}_2)\cdots p(\frac{t}{2}, {\xi}_{k}, y)\rd {\xi}_1\cdots \rd {\xi}_{k}\nn\\
&\ge c_1(2^{-1}t/k)^{-d/\alpha-n/2}\prod^{k-1}_{i=1}\big(c_1(2^{-1}t/k)^{-d/\alpha-n/2}|Q_i|\big)\int_{Q_k}p(\frac{t}{2}, {\xi}_k, y)\rd {\xi}_k\nn\\
&= c_1(2^{-1}t/k)^{-d/\alpha-n/2}\prod^{k-1}_{i=1}\big(c_1(2^{-1}t/k)^{-d/\alpha-n/2}(r_0^2/k^2)^{d/\alpha+n/2}\big)\bP^y(X_{t/2}\in Q_k)\nn\\
&\ge  c_2c_3^{k} t^{-d/\alpha-n/2}\bP^y(X_{t/2}\in Q_k).
\end{align}
Here, $c_1, c_2, c_3$ are positive constants can be chosen independently of $k$. To find the lower bound of $\bP^y(X_{t/2}\in Q_k)$, we follow the proof of \cite[Theorem 4.21]{Xu13}. Since the proofs are the same for $d\ge1$, we only consider the case that $d=1$. 

Let $Q=\cQ(y,3^{-2/\alpha}\eps t^{1/\alpha})$ and $\wt Q_k:=\cQ(z_k, 4^{-1}(r_0/k)^{2/\alpha})\subset Q_k$. Then, for $v\in Q$ and $u\in \wt Q_k$,
\begin{align}\label{lb_ineq}
|v^1-u^1|\le |v^1-x^1|+|x^1-y^1|+|y^1-u^1|\le |x^1-y^1|+\frac{2\eps}{3}t^{1/\alpha}\le 2|x^1-y^1|.
\end{align}
Moreover, by \autoref{l:surv_lb}, we see that 
\begin{align}\label{surv_lb1}
\bE^y[(t/2)\wedge \tau_Q]\ge \frac{t}{2}\,\bP^y(\tau_Q>t/2)\ge c_4t,
\end{align}
and for any $u\in \wt Q_k$,
\begin{align}\label{surv_lb2}
\bP^u(\tau_{Q_k}>t)\ge \bP^u(\tau_{\cQ(u, c_5(r_0/k)^{2/\alpha} )}>t)>C_1\exp\big(-C_2c_5^{-\alpha}36\eps^{-\alpha}k\big),
\end{align}
where $c_5\in(0, 1-4^{-\alpha/2})$, and $C_1, C_2>0$ are the constants in \autoref{l:surv_lb}.
Let $\wt\sigma:=\inf\{t>0: X_t\in \wt Q_k\}$ and $c_6:=C_2c_5^{-\alpha}36\eps^{-\alpha}$. 
By the strong Markov property, \eqref{surv_lb2}, the L\'evy system, \eqref{lb_ineq} and \eqref{surv_lb1},
\begin{align}\label{lower2}
\bP^y(X_{t/2}\in Q_k)&\ge\bP^y\big(\bP^{X_{\wt\sigma}}(\tau_{Q_k}\circ \theta_{\wt \sigma}>t/2-\wt \sigma), \wt\sigma<t/2\big)\nn\\
&\ge C_1\exp\big(-c_6k\big)\bP^y\big(X_{(t/2)\wedge \tau_Q}\in \wt Q_k\big)\nn\\
&=C_1\exp\big(-c_6k\big)\bE^y\left[\int^{(t/2)\wedge \tau_Q}_0\int_{\wt Q_k}J(X_s, u)m(\rd u)\rd s\right]\nn\\
&\ge \kappa^{-1}C_1\exp\big(-c_6k\big)\bE^y\left[\int^{(t/2)\wedge \tau_Q}_0\int_{\wt Q_k}\frac{1}{|X_s-u|^{1+\alpha}}m(\rd u)\rd s\right]\nn\\
&\ge c_7\exp\big(-c_6 k\big)\bE^y[(t/2)\wedge \tau_Q]\frac{(t/k)^{\alpha^{-1}}}{|x^1-y^1|^{1+\alpha}}\nn\\
&\ge c_7c_8\exp\big(-2c_6 k\big)\frac{t^{1+\alpha^{-1}}}{|x^1-y^1|^{1+\alpha}},
\end{align}
where the last inequality follows from that there exists $c_8>0$ such that $e^{c_6k}\ge c_8k^{1/\alpha}$ for all $k\ge1$.
Thus, by \eqref{lower1} and \eqref{lower2},
\begin{align*}
p(t,x,y)
&\ge c_2c_7c_8 t^{-d/\alpha-n/2}\exp\big(-\log(1/c_3)k\big)\exp\big(-2c_6 k\big)\frac{t^{1+\alpha^{-1}}}{|x^1-y^1|^{1+\alpha}}\\
&\ge c_2c_7c_8 t^{-d/\alpha-n/2}\exp\left(-\frac{c_9 |\wt x-\wt y|^2}{t}\right)\frac{t^{1+\alpha^{-1}}}{|x^1-y^1|^{1+\alpha}}.
\end{align*}
\end{proof}

\vspace{5mm}

\section{Off-diagonal upper bound}\label{sec:suhk}

In order to prove \autoref{theo:main} it remains to show the off-diagonal upper bound in \eqref{eq:main-estim}, which is the main goal of this section. We present the proof in a fully self-contained manner in the case $d=n=1$ in \autoref{sec:upper-dn-one}. The strategy is analogous in the general case but the presentation becomes more complex. We treat this case in \autoref{sec:upper-dn-general}. This choice of presentation leads to some redundancy, which we accept for the benefit of higher readability of the main ideas and formal arguments. 

\medskip

Before we explain the method of proof, let us introduce a technical tool that we are going to apply. Recall that $\{P_t, t\ge0\}$ is the transition semigroup of $X$ defined by
$$P_tf(x)=\bE^x[f(X_t)]=\int_{\R^{d+n}}p(t,x,y)f(y)\rd y,$$
for any non-negative Borel function $f$ on $\R^{d+n}$ and for any $t>0$, $x\in\R^{d+n}$. Since $(\cE, \cF)$ is symmetric, the following lemma holds:
\begin{lemma}[{\cite[Lemma 2.1]{BGK09}}]\label{l:2.1}
	Let $U$ and $V$ be two disjoint non-empty open subsets of $\R^{d+n}$ and $f,g$ be non-negative Borel functions on $\R^{d+n}$. Let $\tau=\tau_U$ and $\tau'=\tau_V$ be the first exit times from $U$ and $V$, respectively. Then, for all $a,b,t>0$ such that $a+b=t$, we have
	\begin{align}\label{l:2.1ineq}
	\begin{split}
	\int_{\R^{d+n}} P_tf(x)g(x) \rd x&\le \int_{\R^{d+n}} \bE^{x}[\1_{\{\tau\le a\}}P_{t-\tau}f(X_\tau)] g(x) \rd x\\
	&\qquad+\int_{\R^{d+n}} \bE^{x}[\1_{\{\tau'\le b\}}P_{t-\tau'}g(X_{\tau'})] f(x) \rd x.
	\end{split}
	\end{align}
\end{lemma}

\medskip

The desired upper bound in \eqref{eq:main-estim} will be the final step in an iterative scheme. Let us introduce those conditions that are needed already in the case $d=n=1$. Let $q\ge0$ be given. The we define the following conditions:

\medskip

\noindent{$\HHq{q}{0; 0}$} There exists $C_0 \geq 1$ such that for all $t>0$, $x,y\in\R^{d+n}$,
\begin{align}\label{H0q}
p(t,x,y)&\le C_0t^{-d/\alpha-n/2} \prod_{i=1}^{d}\bigg(1\wedge\frac{t}{|x^i-y^i|^{\alpha}}\bigg)^q.
\end{align}

\noindent{$\HHq{q}{0; n}$} There exist $C_0, c \geq 1$ such that for all $t>0$, $x,y\in\R^{d+n}$,
\begin{align}\label{wtH0q}
p(t,x,y)&\le C_0t^{-d/\alpha-n/2}\exp\bigg(-\frac{|\wt x-\wt y|^2}{c\,t}\bigg) \prod_{i=1}^{d}\bigg(1\wedge\frac{t}{|x^i-y^i|^{\alpha}}\bigg)^q .
\end{align} 

In the case $d=n=1$, the main aim is to prove $\HHq{1+\frac{1}{\alpha}}{0; n}$ which is equivalent to the upper bound in \eqref{eq:main-estim}.

\subsection{The case $d=n=1$}\label{sec:upper-dn-one}{\ }

\medskip

We have already mentioned that the final upper bound in \eqref{eq:main-estim} is the last conclusion in a certain iterative scheme. Let us explain this scheme. 

\medskip

Step 1:
\begin{align*}
&\HHq{0}{0; 0} \hookrightarrow \HHq{\lambda_0}{0; 0} \hookrightarrow \HHq{2 \lambda_0}{0; 0} \ldots \hookrightarrow \HHq{N_0 \lambda_0}{0; 0} \\
\hookrightarrow  &\HHq{1+\alpha^{-1}}{0; 0}=:\HHq{0}{1; 0} \,, 
\end{align*}
where the last definition is natural and facilitates future notation. Every implication within the first of the above chain is a direct application of \autoref{lem:Gl-one}, Part (i). The last implication follows from \autoref{lem:Fl-one}, Part (i).

\medskip

Step 2: Independent from Step 1 we establish $\HHq{0}{0; n}$ with the help of \autoref{theo:uhk}.

\medskip

Step 3: With the help of condition $\HHq{0}{1; 0}$, we establish:
\begin{align*}
&\HHq{0}{0; n} \hookrightarrow \HHq{\lambda_1}{0; n} \hookrightarrow \HHq{2 \lambda_1}{0; n} \ldots \hookrightarrow \HHq{N_1 \lambda_1}{0; n} \\
\hookrightarrow  &\HHq{1+\alpha^{-1}}{0; n}\,,
\end{align*}
where all but the last implication are applications of \autoref{lem:Gl-one}, Part (ii). The last implication follows from \autoref{lem:Fl-one}, Part (ii).

\begin{remark*}
	The presentation above including all implications has the advantage that its generalization to the higher-dimensional case can easily be understood, see \autoref{sec:upper-dn-general}.
\end{remark*}

\medskip

\begin{lemma}\label{lem:Gl-one}
	Assume condition $\HHq{q}{0; \eta}$ holds true for $q<\alpha^{-1}$. Further, assume  
	\begin{itemize}
		\item[(i)] either $\eta=0$, 
		\item[(ii)] or $\eta=1$ and $\HHq{0}{1;0}$ holds true. 
	\end{itemize}
	Then $\HHq{q+\lambda_0}{0; \eta}$ holds true, where $\lambda_0$ depends only on $\alpha$ and satisfies $q+\lambda_0 < 1 + \alpha^{-1}$.
\end{lemma}

\begin{lemma}\label{lem:Fl-one}
	Assume condition $\HHq{q}{0; \eta}$ holds true  $q>\alpha^{-1}$. Further, assume 
	\begin{itemize}
		\item[(i)] either $\eta=0$, 
		\item[(ii)] or $\eta=n$ and $\HHq{0}{1;0}$ holds true. 
	\end{itemize}
	Then $\HHq{1+\alpha^{-1}}{0; \eta}$ holds true.
\end{lemma}

To prove above two lemmas, we need the following technical result, \autoref{prop:main_twodim}. The proof of above two lemmas are given after the proof of \autoref{prop:main_twodim}.

\medskip

\begin{proposition}\label{prop:main_twodim}
Let $\eta\in\{0,1\}$, $\alpha\in(0,2)$ and $t>0$. Assume that $\HHq{q}{0; \eta}$ holds true for some $q\in[0,1+\alpha^{-1}]$. Assume further that either $\eta=0$ or the conjunction $\eta=1$ and $\HHq{0}{1;0}$ hold. Let $t>0$, $x_0=(x_0^1, x_0^2), y_0=(y_0^1, y_0^2)\in\R^2$ with $|x_0^1-y_0^1|\ge \frac52t^{1/\alpha}$. Set $\rho=t^{1/\alpha}$ and take $\theta_1\in\N$ satisfying $\frac54R_1\le |x_0^1-y_0^1|<\frac{10}{4} R_1$, where $R_1=2^{\theta_1}\rho$. Let $\tau=\tau_{\cQ(x_0, R_{1}/8)}$ and $f$ be a non-negative Borel function on $\R^2$ supported in $\cQ(y_0, \rho/8)$. Then there exist $C,c>0$ independent of $x_0, y_0$ and $t$ such that for every $x\in \cQ(x_0, \rho/8)$,
	\begin{align}\label{e:key}
	\begin{split}
	&\bE^{x}\big[\1_{\{\tau\le t/2\}}P_{t-\tau}f(X_\tau)\big]\\
	&\le C t^{-1/\alpha-1/2}\|f\|_1\exp\bigg(-\frac{|x_0^2-y_0^2|^2}{c\,t}\bigg)^{\eta}\cdot\begin{cases}
	\Big(1\wedge\frac{t}{|x_0^1-y_0^1|^{\alpha}}\Big)^{\frac12+q}&\text{if}\;\;q<\alpha^{-1};\\
	\Big(1\wedge\frac{t}{|x_0^1-y_0^1|^{\alpha}}\Big)^{1+\alpha^{-1}}&\text{if}\;\;q>\alpha^{-1}.
	\end{cases}
	\end{split}
	\end{align}
\end{proposition}
Note that there exists $C_3=C_3(\alpha)\ge2$ such that for any $a, b>0$ with $2a\le b$, $a^{\alpha/2}+(b/C_3)^{\alpha/2}\le b^{\alpha/2}$. Thus, for $x\in \cQ(x_0, \rho/8)$, we see that $\cQ(x, R_1/(8C_3))\subset \cQ(x_0, R_1/8)$.  Thus, by \autoref{prop:survival}
\begin{align}\label{exitupper}
\bP^{x}(\tau_{\cQ(x_0, R_1/8)}\le t/2)\le \bP^{x}(\tau_{\cQ(x, R_1/(8C_3))}\le t/2)\le c t R_1^{-\alpha}. 
\end{align}

For given $t>0$ and $k\in\N$, let $\rho=t^{1/\alpha}$ and
\begin{align}\label{def_Dk}
\cD_0:=D_0\times \R:= (-2\rho, 2\rho)\times\R, \quad \cD_k:=D_k\times\R:=[2^k\rho, 2^{k+1}\rho)\times \R.\end{align}
Note that $D_k (k\in\N_0)$ is the same set defined in \cite{KKK22}. 
For given $x_0,y_0\in\R^2$, let $A_k:=y_0+\cD_k$. Let $C_4:=8^{2/\alpha}$ so that $C_4^{-\alpha/2}=\frac18$.
For $x\in \cQ(x_0, \rho/8)$ and $\tau=\tau_{\cQ(x_0, R_1/C_4)}$, set
\begin{align*}
\Phi(k)&=\bE^x\big[\1_{\{\tau\le t/2\}}\1_{\{X_{\tau}\in A_k\}}P_{t-\tau}f(X_{\tau})\big],\quad k\in\N_0.
\end{align*}
Then,
\begin{align*}
\bE^x\big[\1_{\{\tau\le t/2\}}P_{t-\tau}f(X_{\tau})\big]=\sum^{\infty}_{k=0}\bE^x\big[\1_{\{\tau\le t/2\}}\1_{\{X_{\tau}\in A_k\}}P_{t-\tau}f(X_{\tau})\big]=\sum^{\infty}_{k=0}\Phi(k).
\end{align*}

We observe that given $x\in \cQ(x_0, \rho/8)$, $k+1\le \theta_1$,
\begin{align}\label{upperexit}
\bE^x\left[\int^{t/2\wedge \tau}_{0}\int_{I^1_k}\frac{1}{|X^1_s-\ell|^{1+\alpha}}\rd \ell \rd s\right]
\le \frac{ct}{R_1^{1+\alpha}}2^{k}\rho,
\end{align}
where $I^1_k=\{\ell\in\R:|\ell-y_0^1|\in [2^k\rho, 2^{k+1}\rho)\}$.
Indeed, for $w\in\cQ(x_0, R_1/8)$ and $z\in A_k$ (i.e., $|z^1-y_0^1|\in[2^{k}\rho, 2^{k+1}\rho)$) 
\begin{align*}
|w^1-z^1|&\ge |x_0^1-y_0^1|-|w^1-x_0^1|-|z^1-y_0^1|
\ge \frac54R_1 -\frac{R_1}8-2^{k+1}\rho\ge \frac98R_1-2^{\theta_1}\rho=\frac{R_1}8.
\end{align*}

\vspace{5mm}

\begin{proof}[Proof of \autoref{prop:main_twodim}]
	
\subsection*{Case 1: $\eta=0$} 
	We first derive an upper bound for 
	$$P_{t-\tau}f(z)=\int_{\cQ(y_0, \frac{\rho}{C_3})}p({t-\tau}, z,y)f(y)\rd y$$
	for $z\in A_k$ and $t/2\le t-\tau\le t$. Let $y\in \cQ(y_0, \rho/C_3)$ and $z\in A_k$ for $k\ge1$. Then,
	\begin{align*}
	|z^1-y^1|&\ge |z^1-y_0^1|-|y_0^1-y^1|\ge 2^{k}\rho-\rho/C_3\ge \frac122^{k}\rho.
	\end{align*}
	Thus, for $y\in \cQ(y_0, \rho/C_3)$ and $z\in A_k$ for $k\ge0$,
	\begin{align*}
	1\wedge \frac{t}{|z^1-y^1|^{\alpha}}\le 2^{\alpha}\big(1\wedge 2^{-k\alpha}\big).
	\end{align*}
	Thus,  by $\HHq{q}{0; 0}$, for $z\in A_k$ ($k\ge0$),
	\begin{align}\label{upperptf}
	P_{t-\tau}f(z)&=\int_{\cQ(y_0, \frac{\rho}{C_3})}p({t-\tau},z,y)f(y)\rd y\le C t^{-1/\alpha-1 /2}\|f\|_12^{-k\alpha q}\nn\\
	&\le C t^{-1/\alpha-1 /2}\|f\|_1\left(\frac{t}{R_1^{\alpha}}\right)^q2^{(\theta_1-k)\alpha q}.
	\end{align}
	By \eqref{upperexit} and \eqref{upperptf},
	\begin{align*}
	\sum^{\theta_1-1}_{k=0}\Phi(k)
	&\le C t^{-1/\alpha-1 /2}\|f\|_1\left(\frac{t}{R_1^{\alpha}}\right)^q\sum^{\theta_1-1}_{k=0}2^{(\theta_1-k)\alpha q}\bE^x\big[\1_{\{\tau\le t/2\}}\1_{\{X_{\tau}\in A_k\}}\big]\\
	&\le C t^{-1/\alpha-1 /2}\|f\|_1\left(\frac{t}{R_1^{\alpha}}\right)^q\sum^{\theta_1-1}_{k=0}2^{(\theta_1-k)\alpha q}2^{-\theta_1}2^{-\theta_1\alpha}2^k.
	\end{align*}
	If $q< \alpha^{-1}$, then
	\begin{align*}
	\sum^{\theta_1-1}_{k=0}2^{(\theta_1-k)\alpha q}2^{-\theta_1}2^{-\theta_1\alpha}2^k
	= 2^{-\theta_1\alpha}\sum^{\theta_1-1}_{k=0}2^{(\theta_1-k)(\alpha q-1)}
	\le 2^{-\theta_1\alpha}\sum^{\infty}_{l=1}2^{(\alpha q-1)l}
	\le C2^{-\theta_1\alpha}.
	\end{align*}
	If $q> \alpha^{-1}$, then 
	\begin{align*}
	\sum^{\theta_1-1}_{k=0}2^{(\theta_1-k)\alpha q}2^{-\theta_1(1+\alpha)}2^k
	=2^{-\theta_1(1+\alpha-\alpha q)}\sum^{\theta_1-1}_{k=0}2^{-k(\alpha q-1)}\le C2^{-\theta_1\alpha(1+\alpha^{-1}-q)}.
	\end{align*}
	Thus,
	\begin{align}\label{part1}
	\sum^{\theta_1-1}_{k=0}\Phi(k)&\le C t^{-1/\alpha-1 /2}\|f\|_1\left(\frac{t}{R_1^{\alpha}}\right)^q\cdot
	\begin{cases}
	2^{-\theta_1\alpha}&\mbox{if}\;\;q< \alpha^{-1};\\
	2^{-\theta_1\alpha(1+\alpha^{-1}-q)}&\mbox{if}\;\;q> \alpha^{-1}.
	\end{cases}
	\end{align}

	Now, using \eqref{upperptf} and \autoref{prop:survival}
	\begin{align}\label{part2}
	\sum^{\infty}_{k=\theta_1}\Phi(k)
	&\le C t^{-1/\alpha-1 /2}\|f\|_1\left(\frac{t}{R_1^{\alpha}}\right)^q\bP^x\big(\tau\le t/2\big)\sum^{\infty}_{k=\theta_1}2^{(\theta_1-k)\alpha q}\nn\\
	&\le C t^{-1/\alpha-1 /2}\|f\|_1\left(\frac{t}{R_1^{\alpha}}\right)^q2^{-\theta_1\alpha}.
	\end{align}
	Note that for $q\in(\alpha^{-1}, 1+\alpha^{-1})$, we see $0<1+\alpha^{-1}-q<1$. Thus, by \eqref{part1} and \eqref{part2},
	\begin{align*}
	\sum^{\infty}_{k=0}\Phi(k)&\le C t^{-1/\alpha-1 /2}\|f\|_1\left(\frac{t}{R_1^{\alpha}}\right)^q\cdot
	\begin{cases}
	2^{-\theta_1\alpha}&\mbox{if}\;\;q< \alpha^{-1};\\
	2^{-\theta_1\alpha(1+\alpha^{-1}-q)}&\mbox{if}\;\;q> \alpha^{-1}.
	\end{cases}
	\end{align*}
	Since $2^{-\theta_1\alpha}=\frac{t}{R_1^{\alpha}}$, we obtain \autoref{prop:main_twodim} for $\eta=0$ by the above inequality.
	
	\medskip
	
\subsection*{Case 2: $\eta=1$ and $\HHq{0}{1;0}$.} 	

Let $C_5=8/8^{\alpha/2}>1$.
	If $|x_0^2-y_0^2|\le C_5t^{1/2}$, then for any $y\in \cQ(y_0, \rho/C_4)$ and $z\in A_k$ $(k\ge0)$, we have
	\begin{align*}
	\exp\left(-\frac{|z^2-y^2|^2}{ct}\right)\le 1 \le e^{C_5^2/c}\exp\left(-\frac{|x_0^2-y_0^2|^2}{ct} \right).
	\end{align*}
	Thus, the result follows from the same argument as in the {\bf Case 1}.
	Thus, in the following, we only consider the case $|x_0^2-y_0^2|> C_5t^{1/2}$.
	For the rest of the proof, we let $\cQ_0=\cQ(x_0, R_1/8)$ for notational simplicity. By $\HHq{0}{1;0}$, we have that for any $\xi,\zeta\in \R^2$
	\begin{align*}
	p(t, \xi, \zeta)&\le C t^{-1/\alpha-1 /2} \bigg(1\wedge\frac{t}{|\xi^1-\zeta^1|^{\alpha}}\bigg)^{1+\alpha^{-1}}.
	\end{align*} 
	By \autoref{theo:uhk}, we also have
	\begin{align*}
	p(t, \xi, \zeta)\le C t^{-1/\alpha-1 /2}\bigg(1\wedge\frac{t}{|\xi^1-\zeta^1|^{\alpha}}\bigg)^{\frac13}\exp\bigg(-\frac{|\xi^2-\zeta^2|^2}{c\,t}\bigg).
	\end{align*}
	Thus, for $0<\beta<1$, we see that for any $t>0$ and $\xi,\zeta\in \R^2$, 
	\begin{align}\label{suhk1}
	p(t,\xi,\zeta)&=p(t,\xi,\zeta)^{1-\beta}p(t,\xi,\zeta)^{\beta}\nn\\
	&\le C t^{-1/2-1/\alpha}\exp\left(-\frac{(1-\beta)|\xi^2-\zeta^2|^2}{c\,t}\right)\left(1\wedge \frac{t}{|\xi^1-\zeta^1|^{\alpha}}\right)^{\frac{1-\beta}{3}+\beta(1+\alpha^{-1})}.\quad
	\end{align}
	Take $0<\frac{\alpha^{-1}-1/3}{\alpha^{-1}+2/3}<\beta<1$ so that $\beta':=\frac{1-\beta}{3}+\beta(1+\alpha^{-1})>\alpha^{-1}$. Then, we observe that for any $s>0$,
	\begin{align}\label{eq_fin}
	&\int_{|\xi^1-z^1|\ge 0 }s^{-1/\alpha}\left(1\wedge\frac{s}{|\xi^1-z^1|^{\alpha}}\right)^{\beta'}\rd z^1 \nn\\
	&\le \int_{0\le |\xi^1-z^1| < s^{1/\alpha} }s^{-1/\alpha}\,\rd z^1+\int_{ |\xi^1-z^1| \ge s^{1/\alpha} }s^{-1/\alpha}\left(\frac{s}{|\xi^1-z^1|^{\alpha}}\right)^{\beta'}\rd z^1\nn\\
	&\le C+s^{-1/\alpha+\beta'}\int_{s^{1/\alpha}}^{\infty}\frac{1}{u^{\alpha\beta'}}\,\rd u\nn\\
	&\le C.
	\end{align}

	Let $S:=\{z\in \R^2: \frac12|x_0^2-y_0^2|<|x_0^2-z^2|<\frac32|x_0^2-y_0^2|\}$ and $S_k:=S\cap A_k$.
	Then,
	\begin{align}\label{decompose}
	\Phi(k)&=\bE^{x}\big[\1_{\{\tau\le t/2\}}\1_{\{X_{\tau}\in S_k\}}P_{t-\tau}f(X_\tau)\big]+\bE^{x}\big[\1_{\{\tau\le t/2\}}\1_{\{X_{\tau}\in A_k\setminus S_k\}}P_{t-\tau}f(X_\tau)\big]\nn\\
	&=:\Phi_1(k)+\Phi_2(k).
	\end{align}
	
	We first derive the upper bound for $\sum^{\theta_1-1}_{k=0}\Phi(k)$. By $\HHq{q}{0;1}$, we have that for all $k\ge0$,  $w\in A_k$ and $y\in \cQ(y_0, \rho/C_3)$,
	\begin{align*}
	p(t-\tau, w,y)&\le C t^{-1/\alpha-1 /2}\bigg(1\wedge \frac{t}{|w^1-y^1|^{\alpha}}\bigg)^{q} \exp\bigg(-\frac{|w^2-y^2|^2}{c\,t}\bigg)\\
	&\le C t^{-1/\alpha-1 /2}2^{\alpha q}\big(1\wedge 2^{-k\alpha q}\big)\\
	&\le C t^{-1/\alpha-1 /2}2^{-k\alpha q},
	\end{align*}
	which implies
	\begin{align*}
	P_{t-\tau}f(w)=\int_{\cQ(y_0, \frac{\rho}{C_3})}p(t-\tau,w,y)f(y)\rd y
	\le C\|f\|_1t^{-1/\alpha-1 /2}2^{-k\alpha q}.
	\end{align*}
	Thus, 
	\begin{align}\label{upper_I1}
	\Phi_1(k)&=\bE^{x}\big[\1_{\{\tau\le t/2\}}\1_{\{X_{\tau}\in S_k\}}P_{t-\tau}f(X_\tau)\big]\nn\\
	&=\bE^{x}\big[\1_{\{\tau\le t/2\}}\1_{\{X_{\tau-}\in \cQ_0\cap S\}}\1_{\{X_{\tau}\in S_k\}}P_{t-\tau}f(X_\tau)\big]\nn\\
	&\le C\|f\|_1t^{-1/\alpha-1 /2}2^{-k\alpha q}\bE^{x}\big[\1_{\{\tau\le t/2\}}\1_{\{X_{\tau-}\in \cQ_0\cap S\}}\1_{\{X_{\tau}\in S_k\}}\big].
	\end{align}
	Using the L\'evy system, we have
	\begin{align*}
	\bE^{x}\big[\1_{\{\tau\le t\}}\1_{\{X_{\tau-}\in \cQ_0\cap S\}}\1_{\{X_{\tau}\in S_k\}}\big]=
	\bE^{x}\int_0^{\tau\wedge t}\1_{\{X_{s}\in \cQ_0\cap S\}}\int_{S_k}J(X_s, w) m(\rd w) \rd s
	\end{align*}
	Let $f(s,x,y)=\1_{\{\cQ_0\cap S\}}(x)\1_{\{S_k\}}(y)$. Then, 
	$f(s, X_{s-}, X_s)=\1_{\{X_{s-}\in\cQ_0\cap S\}}\1_{\{X_{s}\in S_k\}}(y)=0$ for $s<\tau$ since $S_k\subset \cQ_0^c$. Thus,
	\begin{align*}
	\bE^{x}\left[\sum_{s\le \tau\wedge t}f(s, X_{s-}, X_s)\right]
	&=\bE^{x}\left[\sum_{s\le t}\1_{\{t<\tau\}}f(s, X_{s-}, X_s)\right]+\bE^{x}\left[\sum_{s\le \tau}\1_{\{\tau\le t\}}f(s, X_{s-}, X_s)\right]\\
	&=\bE^{x}\left[\1_{\{\tau\le t\}}f(\tau, X_{\tau-}, X_{\tau})\right]\nn\\
	&=\bE^{x}\left[\1_{\{\tau\le t\}}\1_{\{X_{\tau-}\in \cQ_0\cap S\}}\1_{\{X_{\tau}\in S_k\}}\right].
	\end{align*}
	Thus, by the L\'evy system,
	\begin{align*}
	&\bE^{x}\big[\1_{\{\tau\le t/2\}}\1_{\{X_{\tau-}\in \cQ_0\cap S\}}\1_{\{X_{\tau}\in S_k\}}\big]
	=\bE^{x}\left[\sum_{s\le \tau\wedge(t/2)}f(s, X_{s-}, X_s)\right]\\
	&=\bE^{x}\int_0^{\tau\wedge(t/2)}\int_{S_k}f(s, X_s, w)J(X_s, w) m(\rd w) \rd s\\
	&=\bE^{x}\int_0^{\tau\wedge(t/2)}\1_{\{X_{s}\in \cQ_0\cap S\}}\int_{S_k}J(X_s, w) m(\rd w) \rd s.
	\end{align*}
	Using this, we have that for $k+1\le \theta_1$
	\begin{align}\label{upper_I3}
	\Phi_1(k)&\le C\|f\|_1t^{-1/\alpha-1 /2}2^{-k\alpha q}\bE^{x}\int_0^{\tau\wedge(t/2)}\1_{\{X_{s}\in \cQ_0\cap S\}}\int_{S_k}J(X_s, w) m(\rd w) \rd s\nn\\
	&\le C\|f\|_1t^{-1/\alpha-1 /2}2^{-k\alpha q}\bE^{x}\int_0^{\tau\wedge(t/2)}\1_{\{X_{s}\in \cQ_0\cap S\}}\int_{A^1_k}\frac{1}{|(X_s)^1-w^1|^{1+\alpha}} \rd w^1 \rd s\nn\\
	&\le C\|f\|_1t^{-1/\alpha-1 /2}2^{-k\alpha q}\frac{2^{k}\rho}{R_1^{1+\alpha}}\int_0^{t}\bP^x(X_{s}\in \cQ_0\cap S) \rd s.
	\end{align}
	By \eqref{suhk1} and \eqref{eq_fin},
	\begin{align*}
	\bP^x(X_{s}\in \cQ_0\cap S)&=\int_{\cQ_0\cap S}p(s, x, z) \rd z\nn\\
	&\le C\int_{\cQ_0\cap S}s^{-1/\alpha-1/2}\left(1\wedge\frac{s}{|x^1-z^1|^{\alpha}}\right)^{\beta'} \exp\left(-\frac{|x^2-z^2|^2}{cs}\right)\rd z\nn\\
	&= Cs^{-1/\alpha}\int_{0\le |x^1-z^1| \le R_1 }\left(1\wedge\frac{s}{|x^1-z^1|^{\alpha}}\right)^{\beta'}\rd z^1 \nn\\
	&\qquad\times s^{-1/2}\int_{\frac12|x_0^2-y_0^2|\le |x^2-z^2| \le \frac32|x_0^2-y_0^2| }\exp\left(-\frac{|x^2-z^2|^2}{cs}\right) \rd z^2
	\nn\\
	&\le C\frac{|x_0^2-y_0^2|}{s^{1/2}}\exp\left(-\frac{|x_0^2-y_0^2|^2}{cs}\right),
	\end{align*}
	which yields
	\begin{align}\label{exp_upper}
	\int^t_0\bP^x(X_{s}\in \cQ_0\cap S)\,\rd s \le Ct\exp\left(-\frac{|x_0^2-y_0^2|^2}{2ct}\right).
	\end{align}
	This together with \eqref{upper_I3} implies that for $k+1\le \theta_1$
	\begin{align}\label{upper_Is}
	\Phi_1(k)\le C\|f\|_1t^{-1/\alpha-1 /2}\exp\left(-\frac{|x_0^2-y_0^2|^2}{2ct}\right)2^{-k\alpha q}\frac{t2^{k}\rho}{R_1^{1+\alpha}}.
	\end{align}

	\medskip
	
	On the other hand,  for any $y\in \cQ(y_0, \rho/C_3)$ and $z\in A_k\setminus S_k$, we have
	\begin{align}\label{expII}
	|z^2-y^2|\ge  \frac38|x_0^2-y_0^2|.
	\end{align}
	Indeed, for $|x_0^2-z^2|\ge \frac32|x_0^2-y_0^2|$, we see that
	\begin{align*}
	|z^2-y^2|\ge |x_0^2-z^2|-|x_0^2-y_0^2|-|y_0^2-y^2|\ge \frac12|x_0^2-y_0^2|-(\rho/C_3)^{\alpha/2}\ge \frac38|x_0^2-y_0^2|,
	\end{align*}
	where the last inequality follows from $|x_0^2-y_0^2|\ge t^{1/2}$ and $C_3^{-\alpha/2}=\frac18$. Similarly, for $|x_0^2-z^2|< \frac12|x_0^2-y_0^2|$, we see that
	\begin{align*}
	|z^2-y^2|&\ge |x_0^2-y_0^2|-|x_0^2-z^2|-|y^2-y_0^2|\ge \frac12|x_0^2-y_0^2|-(\rho/C_3)^{\alpha/2}\ge \frac38|x_0^2-y_0^2|.
	\end{align*}
	By \eqref{expII} and $\HHq{q}{0;1}$, we have that for $z\in A_k\setminus S_k$ and $y\in \cQ(y_0, \rho/C_3)$,
	\begin{align*}
	p(t-\tau, z,y)&\le C t^{-1/\alpha-1 /2}\bigg(1\wedge \frac{t}{|z^1-y^1|^{\alpha}}\bigg)^{q} \exp\bigg(-\frac{|z^2-y^2|^2}{c\,t}\bigg)\\
	&\le C t^{-1/\alpha-1 /2}2^{\alpha q}\big(1\wedge 2^{-k\alpha q}\big)\exp\bigg(-\frac{|x_0^2-y_0^2|^2}{c'\,t}\bigg)\\
	&\le C t^{-1/\alpha-1 /2}2^{-k\alpha q}\exp\bigg(-\frac{|x_0^2-y_0^2|^2}{c'\,t}\bigg).
	\end{align*}
	This implies
	\begin{align*}
	P_{t-\tau}f(z)=\int_{\cQ(y_0, \frac{\rho}{C_3})}p({t-\tau}, z,y)f(y)\rd y
	\le C\|f\|_1t^{-1/\alpha-1 /2}2^{-k\alpha q}\exp\bigg(-\frac{|x_0^2-y_0^2|^2}{4c\,t}\bigg).
	\end{align*}
	Thus, by the above inequality and \eqref{upperexit}, for $k+1\le \theta_1$,
	\begin{align}
	\Phi_2(k)&\le C\|f\|_1t^{-1/\alpha-1 /2}2^{-k\alpha q}\exp\bigg(-\frac{|x_0^2-y_0^2|^2}{4c\,t}\bigg)\bE^{x}\big[\1_{\{\tau\le t/2\}}\1_{\{X_{\tau}\in A_k\setminus S_k\}}\big]\label{upper_II}\\
	&\le C\|f\|_1t^{-1/\alpha-1 /2}2^{-k\alpha q}\frac{t2^{k}\rho}{R_1^{1+\alpha}}\exp\bigg(-\frac{|x_0^2-y_0^2|^2}{4c\,t}\bigg).\label{upper_IIs}
	\end{align}
	By \eqref{upper_Is} and \eqref{upper_IIs},
	\begin{align*}
	\sum^{\theta_1-1}_{k=0}\Phi(k)
	&\le C t^{-1/\alpha-1 /2}\|f\|_1\left(\frac{t}{R_1^{\alpha}}\right)^q\exp\bigg(-\frac{|x_0^2-y_0^2|^2}{4ct}\bigg)\sum^{\theta_1-1}_{k=0}2^{-k\alpha q}\frac{t2^{k}\rho}{R_1^{1+\alpha}}\\
	&\le C t^{-1/\alpha-1 /2}\|f\|_1\left(\frac{t}{R_1^{\alpha}}\right)^q\exp\bigg(-\frac{|x_0^2-y_0^2|^2}{4ct}\bigg)\sum^{\theta_1-1}_{k=0}2^{(\theta_1-k)\alpha q}2^{-\theta_1}2^{-\theta_1\alpha}2^k.
	\end{align*}
	If $q< \alpha^{-1}$, then
	\begin{align*}
	\sum^{\theta_1-1}_{k=0}2^{(\theta_1-k)\alpha q}2^{-\theta_1}2^{-\theta_1\alpha}2^k
	= 2^{-\theta_1\alpha}\sum^{\theta_1-1}_{k=0}2^{(\theta_1-k)(\alpha q-1)}
	\le 2^{-\theta_1\alpha}\sum^{\infty}_{l=1}2^{(\alpha q-1)l}
	\le C2^{-\theta_1\alpha}.
	\end{align*}
	If $q> \alpha^{-1}$, then 
	\begin{align*}
	\sum^{\theta_1-1}_{k=0}2^{(\theta_1-k)\alpha q}2^{-\theta_1(1+\alpha)}2^k
	=2^{-\theta_1(1+\alpha-\alpha q)}\sum^{\theta_1-1}_{k=0}2^{-k(\alpha q-1)}\le C2^{-\theta_1\alpha(1+\alpha^{-1}-q)}.
	\end{align*}
	Thus,
	\begin{align}\label{part21}
	\sum^{\theta_1-1}_{k=0}\Phi(k)&\le C t^{-1/\alpha-1 /2}\|f\|_1\left(\frac{t}{R_1^{\alpha}}\right)^q\exp\bigg(-\frac{|x_0^2-y_0^2|^2}{4ct}\bigg)\cdot
	\begin{cases}
	2^{-\theta_1\alpha}&\mbox{if}\;\;q< \alpha^{-1};\\
	2^{-\theta_1\alpha(1+\alpha^{-1}-q)}&\mbox{if}\;\;q> \alpha^{-1}.
	\end{cases}
	\end{align}

	Now, we derive the upper bound for $\sum^{\infty}_{k=\theta_1}\Phi(k)$.
	Using \eqref{upper_II} and \autoref{prop:survival}
	\begin{align}\label{part2II2}
	\sum^{\infty}_{k=\theta_1}\Phi_2(k)
	&\le C t^{-1/\alpha-1 /2}\|f\|_1\left(\frac{t}{R_1^{\alpha}}\right)^q\exp\bigg(-\frac{|x_0^2-y_0^2|^2}{4ct}\bigg)\bP^x\big(\tau\le t/2\big)\sum^{\infty}_{k=\theta_1}2^{(\theta_1-k)\alpha q}\nn\\
	&\le C t^{-1/\alpha-1 /2}\|f\|_1\left(\frac{t}{R_1^{\alpha}}\right)^q\exp\bigg(-\frac{|x_0^2-y_0^2|^2}{4ct}\bigg)2^{-\theta_1\alpha}.
	\end{align}

	To obtain the upper bound for $\sum^{\infty}_{k=\theta_1}\Phi_1(k)$, we first check that there exist $C,c>0$ such that for any $\eps\in(0,1)$
	\begin{align}\label{upper_gaussian}
	\bE^{x}\big[\1_{\{\tau\le t/2\}}\1_{\{X_{\tau-}\in \cQ_0\cap S\}}\big]
	\le C\left(\frac{t}{R_1^{\alpha}}\right)^{\eps}\exp\left(-c(1-\eps)\frac{|x_0^2-y_0^2|^2}{t}\right).
	\end{align}
	Indeed, by \eqref{suhk1}, \eqref{eq_fin} and \eqref{eq:ineq2}, we see that for $r_2:=|x_0^2-y_0^2|$ and $E:=\{z\in\R^2: |x_0^2-z^2|>\frac{r_2}{4}\}$ and $s\le t$,
	\begin{align}\label{key1}
	&\bP^x\big(|X_s^2-x_0^2|>\frac{r_2}{4}\big)\le\bP^x\big(X_t\in E\big)=\int_E p(t,x,z)\rd z\nn\\
	&\le C\int_{(0,\infty)}s^{-1/\alpha}\left(1\wedge\frac{t}{|x^1-z^1|^{\alpha}}\right)^{\beta(1+\alpha^{-1})} \rd z_1\int_{|x_0^2-z^2|> \frac{r_2}{4}}s^{-1/2}\exp\left(-c\frac{|x^2-z^2|^2}{s}\right) \rd z_2\nn\\
	&\le C\exp\left(-c\frac{|x_0^2-y_0^2|^2}{s}\right)\le C\exp\left(-c\frac{|x_0^2-y_0^2|^2}{t}\right).
	\end{align}
	In the third inequality, we used that for $x\in \cQ(x_0, \rho/8)$, $z\in E$ and $r_2> C_5t^{1/2}$, 
	$$|x^2-z^2|\ge |x_0^2-z^2|-|x^2-x_0^2|>|x_0^2-z^2|-\Big(\frac{\rho}{8}\Big)^{\alpha/2}\ge\frac12|x_0^2-z^2|>\frac18|x_0^2-y_0^2|.$$  
	By \eqref{key1} and \cite[Lemma 3.8]{BBCK09}, we have 
	\begin{align*}
	\bP^x\big(\sup_{s\le t/2}|X_s^2-x_0^2|>\frac{r_2}{2}\big)\le C\exp\left(-c\frac{|x_0^2-y_0^2|^2}{t}\right),
	\end{align*}
	which yields
	\begin{align*}
	\bE^{x}\big[\1_{\{\tau\le t/2\}}\1_{\{X_{\tau-}\in \cQ_0\cap S\}}\big]\le \bP^x\big(\sup_{s\le t/2}|X_s^2-x_0^2|>\frac{r_2}{2}\big)\le C\exp\left(-c\frac{|x_0^2-y_0^2|^2}{t}\right).
	\end{align*}
	On the other hand,  by \autoref{prop:survival}, we also have
	\begin{align*}
	\bE^{x}\big[\1_{\{\tau\le t/2\}}\1_{\{X_{\tau-}\in \cQ_0\cap S\}}\big]\le \bP^{x}\big(\tau\le t/2\big)\le C\frac{t}{R_1^{\alpha}}.
	\end{align*}
	Thus, for any $\eps\in(0,1)$
	\begin{align*}
	\bE^{x}\big[\1_{\{\tau\le t/2\}}\1_{\{X_{\tau-}\in \cQ_0\cap S\}}\big]
	&=\bE^{x}\big[\1_{\{\tau\le t/2\}}\1_{\{X_{\tau-}\in \cQ_0\cap S\}}\big]^{\eps}\bE^{x}\big[\1_{\{\tau\le t/2\}}\1_{\{X_{\tau-}\in \cQ_0\cap S\}}\big]^{1-\eps}\\
	&\le  C\Big(\frac{t}{R_1^{\alpha}}\Big)^{\eps}\exp\left(-c(1-\eps)\frac{|x_0^2-y_0^2|^2}{t}\right),
	\end{align*}
	which proves \eqref{upper_gaussian}.
	
	Using \eqref{upper_I1} and \eqref{upper_gaussian}, we have
	\begin{align}\label{part2I2}
	\sum^{\infty}_{k=\theta_1}\Phi_1(k)
	&\le C\|f\|_1t^{-1/\alpha-1 /2}2^{-k\alpha q}\bE^{x}\big[\1_{\{\tau\le t/2\}}\1_{\{X_{\tau-}\in \cQ_0\cap S\}}\1_{\{X_{\tau}\in \cup_k S_k\}}\big]\nn\\
	&\le C\|f\|_1t^{-1/\alpha-1 /2}2^{-k\alpha q}\bE^{x}\big[\1_{\{\tau\le t/2\}}\1_{\{X_{\tau-}\in \cQ_0\cap S\}}\big]\nn\\
	&\le C t^{-1/\alpha-1 /2}\|f\|_1\left(\frac{t}{R_1^{\alpha}}\right)^{q+\eps}\exp\left(-(1-\eps)\frac{|x_0^2-y_0^2|^2}{ct}\right).
	\end{align}
	Thus, by \eqref{part2II2} and \eqref{part2I2}, we obtain 
	\begin{align}\label{part22}
	\sum^{\infty}_{k=\theta_1}\Phi(k)\le C t^{-1/\alpha-1 /2}\|f\|_1\left(\frac{t}{R_1^{\alpha}}\right)^{q+\eps}\exp\bigg(-(1-\eps)\frac{|x_0^2-y_0^2|^2}{ct}\bigg).
	\end{align}
	Finally, by \eqref{part21}, \eqref{part22} and the fact that $0<1+\alpha^{-1}-q<1$ for $q\in(\alpha^{-1}, 1+\alpha^{-1})$, we have
	\begin{align*}
	\sum^{\infty}_{k=0}\Phi(k)&\le C t^{-1/\alpha-1 /2}\|f\|_1\exp\bigg(-\frac{|x_0^2-y_0^2|^2}{4ct}\bigg)\cdot
	\begin{cases}
	\left(\frac{t}{R_1^{\alpha}}\right)^{q+\frac12}&\mbox{if}\;\;q< \alpha^{-1};\\
	\left(\frac{t}{R_1^{\alpha}}\right)^{1+\alpha^{-1}}&\mbox{if}\;\;q> \alpha^{-1}.
	\end{cases}
	\end{align*}
	Indeed, for $q\in(\alpha^{-1}, 1+\alpha^{-1})$, we take $\eps=1+\alpha^{-1}-q\in(0,1)$.\end{proof}

\medskip

\begin{proof}[Proof of \autoref{lem:Gl-one} and \autoref{lem:Fl-one}]
Let $\rho=t^{1/\alpha}$ and $x_0, y_0\in\R^2$. If $|x_0^1-y_0^1|< \frac52\rho$, then the results follows by \autoref{theo:uhk}. Thus, we assume that $|x_0^1-y_0^1|\ge \frac52\rho$. Consider non-negative Borel functions $f, g$ on $\R^d$ supported in $\cQ(y_0, \frac{\rho}{8})$ and $\cQ(x_0, \frac{\rho}{8})$, respectively.  
We apply 
\autoref{l:2.1} with functions $f, g$,  subsets $U:=\cQ(x_0, s), V:=\cQ(y_0, s)$ for some $s>0$, $a=b=t/2$ and $\tau=\tau_{U}, \tau^{'}=\tau_{V}$.
The first term of the right hand side of \eqref{l:2.1ineq} is
\begin{align*}
\left\langle \E^{\cdot}\left[\1_{\{\tau\le t/2\}}P_{t-\tau}f(X_{\tau})\right], g\right\rangle=\int_{\cQ(x_0, \frac{\rho}{8})}\E^{x}\left[\1_{\{\tau\le t/2\}}P_{t-\tau}f(X_{\tau})\right]\,g(x) \rd x ,
\end{align*}
and a similar identity holds for the second term. Let $U:=\cQ(x_0, 2^{\theta_1}\rho)$, where $\theta_1\in\N$ satisfying $\frac542^{\theta_1}\rho\le |x_0^1-y_0^1|<\frac{10}{4} 2^{\theta_1}\rho$. Then, by \autoref{prop:main_twodim}, i.e., by \eqref{e:key}, 
\begin{align*}
&\left\langle\E^{\cdot}\left[\1_{\{\tau\le t/2\}}P_{t-\tau}f(X_{\tau})\right], g\right\rangle\nn\\
&\le  \,C t^{-1/\alpha-1/2} \|f\|_1\|g\|_1  \exp\bigg(-\frac{|x_0^2-y_0^2|^2}{c\,t}\bigg)^{\eta}\cdot\begin{cases}
	\Big(1\wedge\frac{t}{|x_0^1-y_0^1|^{\alpha}}\Big)^{\frac12+q}&\text{if}\;\;q<\alpha^{-1};\\
	\Big(1\wedge\frac{t}{|x_0^1-y_0^1|^{\alpha}}\Big)^{1+\alpha^{-1}}&\text{if}\;\;q>\alpha^{-1}.
	\end{cases}
\end{align*}
Similarly we obtain the second term of right hand side of \eqref{l:2.1ineq} and therefore,
\begin{align*}
\left\langle P_tf, g\right\rangle\le Ct^{-1/\alpha-1/2} \|f\|_1\|g\|_1\exp\bigg(-\frac{|x_0^2-y_0^2|^2}{c\,t}\bigg)^{\eta}\cdot\begin{cases}
	\Big(1\wedge\frac{t}{|x_0^1-y_0^1|^{\alpha}}\Big)^{\frac12+q}&\text{if}\;\;q<\alpha^{-1};\\
	\Big(1\wedge\frac{t}{|x_0^1-y_0^1|^{\alpha}}\Big)^{1+\alpha^{-1}}&\text{if}\;\;q>\alpha^{-1}.
	\end{cases}
\end{align*} 
Since $P_tf(x)=\int_{\Rd} p(t, x, y)f(y)\rd y$ and $p$ is a continuous function, we obtain the following estimate: for $t>0$ and $x_0, y_0\in\R^2$,
 \begin{align*}
 p(t, x_0,y_0)
& \le C t^{-1/\alpha-1/2}\exp\bigg(-\frac{|x_0^2-y_0^2|^2}{c\,t}\bigg)^{\eta}\cdot\begin{cases}
	\Big(1\wedge\frac{t}{|x_0^1-y_0^1|^{\alpha}}\Big)^{\frac12+q}&\text{if}\;\;q<\alpha^{-1};\\
	\Big(1\wedge\frac{t}{|x_0^1-y_0^1|^{\alpha}}\Big)^{1+\alpha^{-1}}&\text{if}\;\;q>\alpha^{-1}.
	\end{cases}
 \end{align*}
This proves  \autoref{lem:Gl-one} and \autoref{lem:Fl-one}.
\end{proof}

We have established the upper bound in \eqref{eq:main-estim} in the case $d=n=1$.

\subsection{Strategy in the general case}\label{sec:upper-dn-general} In the general case, further to the two conditions $\HHq{q}{0; 0}$, $\HHq{q}{0; n}$, we need two more conditions. To this end, assume $q\ge0$ and $l\in\{1,\dots, d-1\}$ be given. Then we define two new conditions as follows:

\medskip

\noindent{$\HHq{q}{l; 0}$} There exists $C_0 \geq 1$  such that for all $t>0$, $x,y\in\R^{d+n}$ with $|x^{1}-y^{1}|\le \dots\le |x^{d}-y^{d}|$ the following holds:
\begin{align}\label{Hlq}
p(t,x,y)&\le C_0t^{-d/\alpha-n/2} \prod^{d-l}_{i=1}\bigg(1\wedge\frac{t}{|x^{i}-y^{i}|^{\alpha}}\bigg)^q\prod^{d}_{i=d-l+1}\bigg(1\wedge\frac{t}{|x^{i}-y^{i}|^{\alpha}}\bigg)^{1+\alpha^{-1}}.
\end{align} 

\noindent{$\HHq{q}{l; n}$} There exist $C_0, c \geq 1$ such that for all $t>0$, $x,y\in\R^{d+n}$ with $|x^{1}-y^{1}|\le \dots\le |x^{d}-y^{d}|$ the following holds:
\begin{align}\label{wtHlq}
p(t,x,y)&\le C_0t^{-d/\alpha-n/2}\exp\bigg(-\frac{|\wt x-\wt y|^2}{c\,t}\bigg)\nn\\
&\qquad\times \prod^{d-l}_{i=1}\bigg(1\wedge\frac{t}{|x^{i}-y^{i}|^{\alpha}}\bigg)^q\prod^{d}_{i=d-l+1}\bigg(1\wedge\frac{t}{|x^{i}-y^{i}|^{\alpha}}\bigg)^{1+\alpha^{-1}} .
\end{align} 
Our overall aim of this section is to prove $\HHq{1+\alpha^{-1}}{d-1; n}$, which is equivalent to the upper bound in \eqref{eq:main-estim}. See \autoref{lem:reduction_cases} below.

\medskip

We have already mentioned that the final upper bound in \eqref{eq:main-estim} is the last conclusion in a certain iterative scheme. Let us explain this scheme. 

\medskip

Step 1:
\begin{alignat*}{4}
&\HHq{0}{0; 0} &&\hookrightarrow \HHq{\lambda_0}{0; 0} &&\hookrightarrow \HHq{2 \lambda_0}{0; 0} \ldots &&\hookrightarrow \HHq{N_0\lambda_0}{0; 0}\\
\hookrightarrow &\HHq{0}{1; 0} &&\hookrightarrow \HHq{\lambda_1}{1; 0} &&\hookrightarrow \,\, \ldots  \ldots \ldots \,\,  &&\hookrightarrow \HHq{N_1\lambda_1}{1; 0}\\
&&&\vdots&&&\vdots&\\
\hookrightarrow &\HHq{0}{d-1; 0} &&\hookrightarrow \HHq{\lambda_{d-1}}{d-1; 0} &&\hookrightarrow \ldots  \ldots \ldots &&\hookrightarrow \HHq{N_{d-1}\lambda_{d-1}}{d-1; 0}\\
\hookrightarrow &\HHq{1+\alpha^{-1}}{d-1; 0}.
\end{alignat*}
Every implication within each line of the above chain is a direct application of \autoref{lem:Gl-gen}, Part (i). The implication from the last condition in one line to the first condition in the next line follows from \autoref{lem:Fl-gen}, Part (i).

\medskip

Step 2: Independent from Step 1 we establish $\HHq{0}{0; n}$ with the help of \autoref{theo:uhk}.

\medskip

Step 3: With the help of condition $\HHq{0}{d; 0}:=\HHq{1+\alpha^{-1}}{d-1; 0}$, we establish:
\begin{alignat*}{4}
&\HHq{0}{0; n} &&\hookrightarrow \HHq{\lambda_0}{0; n} &&\hookrightarrow \HHq{2 \lambda_0}{0; n} \ldots &&\hookrightarrow \HHq{N_0\lambda_0}{0; n}\\
\hookrightarrow &\HHq{0}{1; n} &&\hookrightarrow \HHq{\lambda_1}{1; n} &&\hookrightarrow \,\, \ldots  \ldots \ldots \,\,  &&\hookrightarrow \HHq{N_1\lambda_1}{1; n} \\
&&&\vdots&&&\vdots&\\
\hookrightarrow &\HHq{0}{d-1; n} &&\hookrightarrow \HHq{\lambda_{d-1}}{d-1; n} &&\hookrightarrow \ldots  \ldots \ldots &&\hookrightarrow \HHq{N_{d-1}\lambda_{d-1}}{d-1; n}\\
\hookrightarrow &\HHq{1+\alpha^{-1}}{d-1; n}.
\end{alignat*}
Every implication within each line of the above chain is a direct application of \autoref{lem:Gl-gen}, Part (ii). The implication from the last condition in one line to the first condition in the next line follows from \autoref{lem:Fl-gen}, Part (ii).

\medskip

Note that by \autoref{theo:uhk}  we have $\HHq{\alpha/3}{0; n}$. Thus, the upper bound in \eqref{eq:main-estim} holds true for $t>0$, $x,y\in\R^{d+n}$ with  $\max_{i\in\{1,2,\dots,d\}}|x^i-y^i|\le \frac52 t^{1/\alpha}$. Hence it suffices to assume that one of the values $|x^i-y^i|$ is larger than $\frac52 t^{1/\alpha}$, see \autoref{def:theta_and_R} below.

\medskip

We can obtain the sharp upper bound in \eqref{eq:main-estim} by applying the following lemmas.

\begin{lemma}\label{lem:Gl-gen}
Assume condition $\HHq{q}{l; \eta}$ holds true for $l\in\{0, \ldots, d-1\}$, $q<\alpha^{-1}$. Further, assume  
\begin{itemize}
	\item[(i)] either $\eta=0$, 
	\item[(ii)] or $\eta=n$ and $\HHq{0}{d;0}$ holds true. 
\end{itemize}
Then $\HHq{q+\lambda_l}{l; \eta}$ holds true, where $\lambda_l>0$ depends only on $l$, $\alpha$ and satisfies $q+\lambda_l < 1 + \alpha^{-1}$.
\end{lemma}

\begin{lemma}\label{lem:Fl-gen}
Assume condition $\HHq{q}{l; \eta}$ holds true for $l\in\{0, \ldots, d-1\}$ and  $q>\alpha^{-1}$. Further, assume 
\begin{itemize}
	\item[(i)] either $\eta=0$, 
	\item[(ii)] or $\eta=n$ and $\HHq{0}{d;0}$ holds true. 
\end{itemize}
Then $\HHq{0}{l+1; \eta}$ holds true, where $\HHq{0}{d; \eta}:=\HHq{1+\alpha^{-1}}{d-1; \eta}$.
\end{lemma}

\begin{definition}\label{def:theta_and_R}
	Let $x_0, y_0 \in \R^{d+n}$ satisfy $|x_0^{i}-y_0^{i}|\le |x_0^{i+1}-y_0^{i+1}|$ for every $i\in\{1,2,\dots, d-1\}$. Let $t>0$, set $\rho:=t^{1/\alpha}$. For $i\in \{1, \dots, d\}$ define $\lengthR_i\in \Z$ and $R_i > 0$ such that 
	\begin{align}\label{d:thi_Ri}
	\frac{5}{4}2^{\lengthR_i} \le  \frac{|x_0^{i}-y_0^{i}|}{\rho}  < \frac{10}{4}2^{\lengthR_i} \qquad\mbox{and}\qquad	R_i=2^{\lengthR_i}\rho\,.
	\end{align}
	Then $\lengthR_i\le \lengthR_{i+1}$ and $R_i\le R_{i+1}$. We say that condition $\mathcal{R}(i_0)$ holds if 
	\begin{align}\label{eq:case-R-i0}
	\lengthR_1\le \ldots \le \lengthR_{i_0-1}\le 0<1\le \lengthR_{i_0}\le \ldots\le \lengthR_d\, 
	\tag*{$\mathcal{R}(i_0)$} .
	\end{align}
	We say that condition $\mathcal{R}(d+1)$ holds if $\lengthR_1\le \ldots \le \lengthR_d \leq 0 < 1$.
\end{definition}

\begin{lemma}\label{lem:reduction_cases}
Let $t > 0$ and $x_0, y_0$ be two points in $\R^{d+n}$ satisfying condition $\mathcal{R}(i_0)$ for $i_0 \in \{d-l+1, \ldots, d+1\}$.  Assume condition $\HHq{q}{l; \eta}$ holds for some $\eta\in\{0,n\}$, $l \in \{0,\ldots, d-1\}$ and $q \geq 0$. Then 
\begin{align}\label{eq:reduction_cases}
p(t,x_0, y_0) \le C t^{-d/\alpha-n/2}\exp\bigg(-\frac{|\wt{x_0}-\wt{y_0}|^2}{c\,t}\bigg)^{\eta/n}\prod_{i=1}^{d} \left(\frac{t}{|x_0^{i}-y_0^{i}|^\alpha}\wedge 	1\right)^{1+\alpha^{-1}}
\end{align}
holds for some constants $C, c > 0$ independent of $t$ and $x_0, y_0$.
\end{lemma}
\begin{proof}
The proof is the same with that of \cite[Lemma 3.2]{KKK22}. Thus, we skip the proof.

\end{proof}

\medskip

For the proof of \autoref{lem:Gl-gen} and \autoref{lem:Fl-gen}, we use the following which is a key result in proving sharp upper bound.

\begin{proposition}\label{prop:main_general}
Let $\eta\in\{0,n\}$ and $\alpha \in (0,2)$. Assume that  $\HHq{q}{l; \eta}$ holds true for some $l\in \{0,1,\ldots, d-1\}$,  $q \in [0,1 + \alpha^{-1}]$. Assume further that either $\eta=0$ or the conjunction  $\eta=n$ and $\HHq{0}{d;0}$ hold.
Let $t > 0$, set $\rho=t^{1/\alpha}$. Consider  $x_0, y_0\in \R^{d+n}$ satisfying the condition $\mathcal{R}(i_0)$ for some $i_0 \in \{1, \dots, d-l\}$, and let $R_j=2^{\theta_j}\rho$ as defined in \eqref{d:thi_Ri}. Let $f$ be a non-negative Borel function on $\R^{d+n}$ supported in $B(y_0, \tfrac{\rho}{8})$.
Let $j_0 \in \{i_0,\ldots, d-l\}$ and define an exit time $\tau$ by $\tau=\tau_{B(x_0, {R_{j_0}}/{8})}$.
Then there exist $C, c>0$ independent of  $x_0, y_0$ and $t$ such that for every $x\in B(x_0, \frac{\rho}{8})$,
	\begin{align}\label{eq:main1}
	\begin{split}
	&\E^{x}\left[\1_{\{\tau\le t/2\}}P_{t-\tau}f(X_{\tau})\right]\\
	&\le \,C t^{-d/\alpha-n/2} \|f\|_1\exp\bigg(-\frac{|\wt {x_0}-\wt {y_0}|^2}{c\,t}\bigg)^{\eta/n} \prod_{j=d-l+1}^d\left(\frac{t}{|x_0^{j}-y_0^{j}|^{\alpha}} \wedge 1 \right)^{1+\alpha^{-1}} \\
	& \qquad \times\prod_{j=j_0+1}^{d-l}
	\left(\frac{t}{|x_0^{j}-y_0^{j}|^\alpha} \wedge 1 \right)^q \cdot 
	\begin{cases}
	\left(\tfrac{t}{| x_0^{j_0}-y_0^{j_0}|^\alpha} \wedge 1 \right)^{\frac12+q}
	&\mbox{ if } q<\alpha^{-1}\\
	\left(\tfrac{t}{|x_0^{j_0}-y_0^{j_0}|^\alpha} \wedge 1 \right)^{1+\alpha^{-1}}
	&\mbox{ if } q>\alpha^{-1}.
	\end{cases}
	\end{split}
	\end{align}
\end{proposition}

\medskip

\subsection{Proof of \autoref{prop:main_general}, \autoref{lem:Gl-gen} and \autoref{lem:Fl-gen}}\label{s:proof}{\ }

\medskip

In this section we will explain how the proofs of \autoref{prop:main_general}, \autoref{lem:Gl-gen} and \autoref{lem:Fl-gen} can be derived along the same lines as in \cite{KKK22}. Since the mains ideas have been already been demonstrated in \autoref{sec:upper-dn-one} when considering the special case $d=n=1$, we here limit ourselves to those parts of the proofs where the application of \cite{KKK22} is not direct. 

\medskip

\begin{proof}[Proof of \autoref{prop:main_general}]
Let $l\in\{0,1,\dots, d-l\}$, $t>0$ and $x_0,y_0\in\R^{d+n}$ satisfy $\cR(i_0)$ for some $i_0 \in \{1, \dots, d-l\}$. For $d>1$, we define subset $\cD_k$ of $\R^{d+n}$ as follows: for $k\in\N_0$ and $\rho=t^{1/\alpha}$,
\begin{align}\label{def_Dk_gen}
\cD_k:=D_k\times\R^n \,,
\end{align}
where $D_k^{\gamma, \switch}, D_k^{\gamma}, D_k (\subset \R^d)$ are the same set defined in \cite{KKK22}.  Using $\cD_k$, we define
$$A_k=y_0+\rho \cD_k.$$
Then, it is easy to see that $\cup_{k=0}^{\infty} A_k=\R^{d+n}$.
For $k\in\N_0$, $j_0 \in \{i_0, \dots, d\}$ and $s(j_0)=R_{j_0}/8$, set 
\begin{align*}
A_k^{i}&:=A_k^{i}(j_0):=(y_0+\cD_k)\cap \bigcup_{u\in\overline{\cQ(x_0, s(j_0)/8)}}\{u+h\e^i|h\in\R\}.
\end{align*}
Let $S:=\{z\in \R^{d+n}: \frac12|\wt{x_0}-\wt{y_0}|<|\wt{x_0}-\wt z|<\frac32|\wt{x_0}-\wt{y_0}|\}$ and $S_k^i:=S\cap A_k^i$.
Let $C_3:=8^{2/\alpha}$. Then, $C_3^{-\alpha/2}=\frac18$.
For $x\in \cQ(x_0, \rho/8)$ and $\tau=\tau_{\cQ(x_0, s(j_0)/C_3)}$, set
\begin{alignat*}{2}
\Phi(k)&=\bE^x\big[\1_{\{\tau\le t/2\}}\1_{\{X_{\tau}\in A_k\}}P_{t-\tau}f(X_{\tau})\big],\quad&&\mbox{for}\;\; k\in\N_0,\\
\Phi^i(k)&=\bE^x\big[\1_{\{\tau\le t/2\}}\1_{\{X_{\tau}\in A_k^{i}\}}P_{t-\tau}f(X_{\tau})\big],\quad&&\mbox{for}\;\; k\in\N_0,\;i\in\{1,2,\dots, d\},\\
\Phi^i_1(k)&=\bE^x\big[\1_{\{\tau\le t/2\}}\1_{\{X_{\tau}\in S_k^{i}\}}P_{t-\tau}f(X_{\tau})\big],\quad&&\mbox{for}\;\; k\in\N_0,\;i\in\{1,2,\dots, d\},\\
\Phi^i_2(k)&=\bE^x\big[\1_{\{\tau\le t/2\}}\1_{\{X_{\tau}\in A_k^{i}\setminus S\}}P_{t-\tau}f(X_{\tau})\big],\quad&& \mbox{for}\;\;k\in\N_0,\;i\in\{1,2,\dots, d\},\\
\Phi_1(0)&=\bE^x\big[\1_{\{\tau\le t/2\}}\1_{\{X_{\tau}\in S_0\}}P_{t-\tau}f(X_{\tau})\big],\\
\Phi_2(0)&=\bE^x\big[\1_{\{\tau\le t/2\}}\1_{\{X_{\tau}\in A_0\setminus S\}}P_{t-\tau}f(X_{\tau})\big].
\end{alignat*}
Then, we can write 
\begin{align*}
\bE^{x}\left[\1_{\{\tau\le t/2\}}P_{t-\tau}f(X_{\tau})\right]
&=\sum_{k=0}^{\infty}\Phi(k)
=\sum_{k=1}^{\infty}\left(\sum_{i=1}^{d}\big(\Phi^i_1(k)+\Phi^i_2(k)\big)\right)+\Phi_1(0)+\Phi_2(0)\nn\\
&=\left(\sum_{k=1}^{\infty}\Phi^d_1(k) + \sum_{k=1}^{\infty}\Phi^{d-1}_1(k)  +\dots+   \sum_{k=1}^{\infty}\Phi^1_1(k)+\Phi_1(0)\right) \\
&\quad\quad   +\left(\sum_{k=1}^{\infty}\Phi^d_2(k) + \sum_{k=1}^{\infty}\Phi^{d-1}_2(k)  +\dots+   \sum_{k=1}^{\infty}\Phi^1_2(k)+\Phi_2(0)\right)\nn\\
&=:\cM_1+\cM_2.
\end{align*}
For $a,b \in \N$, set
$\TT{a}{b}:=\Tj{a}{b}$. Then, by \cite[Remark 4.4]{KKK22}, for $a=1,2,$
\begin{align} \label{eq:main}
\begin{split}
\cM_a 
&=  \sum_{k=\TT{j_0}{d}-\lengthR_d}^{\infty}\Phi^d_a(k) + \sum_{k=\TT{j_0}{d}-\lengthR_{d-1}}^{\infty}\Phi^{d-1}_a(k) + \ldots + \sum_{k=\TT{j_0}{d}-\lengthR_{j_0}}^{\infty}\Phi^{j_0}_a(k)\\
& \qquad +\sum_{k=\TT{j_0}{d}}^{\infty}\Big(\Phi^{j_0-1}_a(k)+\ldots+\Phi^1_a(k)\Big)+\Phi_a(0) \\
&=\sum_{i={j_0}}^d\mathcal{S}_a(i)+\sum_{i={1}}^{j_0-1}\mathcal{T}_a(i)+\Phi_a(0),
\end{split}
\end{align}
where $\mathcal{S}_a(i):=\sum_{k=\TT{j_0}{d}-\lengthR_i}^{\infty}\Phi^i_a(k)$
and $\mathcal{T}_a(i):=\sum_{k=\TT{j_0}{d}}^{\infty}\Phi^i_a(k)$.

\medskip

We will find the upper bounds of $\mathcal{S}_a(i)$, $\mathcal{T}_a(i)$ and $\Phi_a(0)$ for $a=1, 2$. 
We first consider the case that $\eta=0$.
Under the condition $\HHq{q}{l; 0}$, we follow the proofs in \cite{KKK22} to obtain that for $a=1,2$, $\mathcal{S}_a(i)$, $\mathcal{T}_a(i)$ and $\Phi_a(0)$ are bounded above by the right hand side of \eqref{eq:main1}. Thus, we obtain \autoref{prop:main_general} for the case that $\eta=0$. 

Next, we consider the case that $\eta=n$ and $\HHq{0}{d; 0}$ hold. Then, by \autoref{lem:reduction_cases}, we obtain the following rough heat kernel upper bound:
\begin{align}\label{rough_hke_gen}
p(t, \xi, \zeta)&\le C t^{-d/\alpha-n /2}\prod^{d}_{i=1} \bigg(1\wedge\frac{t}{|\xi^i-\zeta^i|^{\alpha}}\bigg)^{1+\alpha^{-1}}\quad \text{for all}\;\;\xi,\zeta\in \R^{d+n}.
\end{align} 
By \eqref{rough_hke_gen} and \autoref{theo:uhk}, we see that for any $\beta\in(0,1)$
	\begin{align*}
	p(t,\xi,\zeta)&=p(t,\xi,\zeta)^{1-\beta}p(t,\xi,\zeta)^{\beta}\nn\\
	&\le C t^{-n/2-d/\alpha}\exp\Big(-\frac{(1-\beta)|\wt\xi-\wt\zeta|^2}{c\,t}\Big)\prod^{d}_{i=1}\left(1\wedge \frac{t}{|\xi^i-\zeta^i|^{\alpha}}\right)^{\frac{1-\beta}{3}+\beta(1+\alpha^{-1})},
	\end{align*}
which is a key observation to deal with $\mathcal{S}_1(i)$, $\mathcal{T}_1(i)$ and $\Phi_1(0)$. By this and the condition $\HHq{q}{l; n}$, we can follow the argument in \cite{KKK22} and apply the method of proof for the case $d=n=1$ to obtain upper bounds for $\mathcal{S}_1(i)$, $\mathcal{T}_1(i)$ and $\Phi_1(0)$. For $\mathcal{S}_2(i)$, $\mathcal{T}_2(i)$ and $\Phi_2(0)$, we just follow the argument in \cite{KKK22} as in the case of $\eta=0$. Then, we obtain the desired upper bounds.
\end{proof}

\medskip

\begin{proof}[Proof of \autoref{lem:Gl-gen} and \autoref{lem:Fl-gen}]
\autoref{prop:main_general} allows us to deduce \autoref{lem:Gl-gen} and \autoref{lem:Fl-gen} in the same way as the corresponding lemmas are deduced in \cite{KKK22}, see the proofs of Lemma 2.5, Lemma 2.6 and Lemma 2.7 therein.
\end{proof}


\end{document}